\documentclass[review]{elsarticle}

\usepackage[bookmarksnumbered, plainpages, colorlinks]{hyperref}

\usepackage{algorithm}
\usepackage{algorithmic}
\usepackage{amsmath}
\usepackage{amssymb}
\usepackage{float}

\usepackage[top=2cm, bottom=2cm, left=2cm, right=2cm]{geometry}
\usepackage{amsmath}
\usepackage{amssymb}
\usepackage{float}
\usepackage[american]{babel}
\usepackage{microtype}
\usepackage{graphicx, latexsym, epsfig, amssymb, mathrsfs}
\usepackage{enumerate, xspace, amscd, amsfonts, graphicx, color,enumitem}
\usepackage{amscd, verbatim, amsmath}
\usepackage{times}
\usepackage{amsthm}
\usepackage{subfigure}
\graphicspath{{figures/}}

\newtheorem{theorem}{Theorem}[section]

\newtheorem{remark}{Remark}[section]

\newtheorem{lemma}{Lemma}[section]
\newtheorem{corollary}{Corollary}[section]
\newtheorem{example}{Example}[section]
\numberwithin{equation}{section}

\begin{document}

\begin{frontmatter}

\title{Strong convergence of an inertial Tseng's extragradient algorithm for  pseudomonotone variational inequalities with applications to optimal control problems}


\author[mymainaddress]{Bing Tan}
\ead{bingtan72@gmail.com}


\author[mythirdaddress]{Xiaolong Qin\corref{mycorrespondingauthor}}
\cortext[mycorrespondingauthor]{Corresponding author}
\ead{qxlxajh@163.com}

\address[mymainaddress]{Institute of Fundamental and Frontier Sciences, University of Electronic Science and Technology of China, Chengdu, China}
\address[mythirdaddress]{Department of Mathematics, Zhejiang Normal University, Zhejiang, China}

\begin{abstract}
We investigate an inertial viscosity-type Tseng's extragradient algorithm with a new step size to solve  pseudomonotone variational inequality problems in real Hilbert spaces. A strong convergence theorem of the algorithm is obtained without the prior information of the Lipschitz constant of the operator and also without any requirement of additional projections. Finally, several computational tests are carried out to demonstrate the reliability and benefits of the algorithm and compare it with the existing ones. Moreover, our algorithm is also applied to solve the variational inequality problem that appears in optimal control problems. The algorithm presented in this paper improves some known results in the literature.
\end{abstract}

\begin{keyword}
Pseudomonotone variational inequality \sep Tseng's extragradient method \sep inertial method \sep viscosity method \sep optimal control problem
\MSC[2010] 47H05 \sep 47H09 \sep 65K15   \sep 47J20
\end{keyword}

\end{frontmatter}

\section{Introduction}
The goal of this study is to investigate a fast iterative method for discovering a solution to the variational inequality problem (in short, VIP).  In this paper, one always assumes that $ H $ is a real Hilbert space with $\langle\cdot, \cdot\rangle$ and the induced norm $\|\cdot\|$, and $ C $ is a  closed and convex nonempty subset in $ H $. Let us first elaborate on the issues involved in this research as follows:
\begin{equation*}\label{VIP}
\text{find } y^{*} \in C \text { such that }\langle \mathcal{A} y^{*}, z-y^{*}\rangle \geq 0,\quad \forall z \in C\,,\tag{VIP}
\end{equation*}
where $ \mathcal{A}: H \rightarrow H $ is a  nonlinear mapping. We denote the solution set of~\eqref{VIP} as $ \mathrm{VI}(C,\mathcal{A}) $.

Variational inequalities are powerful tools and models in applied mathematics and act an essential role in society, optimization, economics, transportation, mathematical programming,  engineering mechanics, and other fields (see, for instance, \cite{QA,SYVS,AIY}).  In the last decades, various effective solution methods have been investigated and developed to solve the problems of type \eqref{VIP}; see, e.g., \cite{Cho2,SIna,tanjnca} and the references therein. It should be pointed out that these approaches usually require that   mapping $ \mathcal{A} $ has certain monotonicity. In this paper, we consider that the mapping $ \mathcal{A} $  associated with \eqref{VIP} is pseudomonotone (see the definition below), which is a broader category than monotone mappings.

Let us review some nonlinear mappings in nonlinear analysis for further use. For any elements $ p, q \in H $, one recalls that a mapping $\mathcal{A}: {H} \rightarrow {H}$ is said to be:
\begin{enumerate}[label=(\arabic*)]
\item  \emph{$\eta$-strongly monotone} if there is a positive number $ \eta $ such that
\[
\langle \mathcal{A} p-\mathcal{A} q, p-q\rangle \geq \eta\|p-q\|,
\]
\item \emph{$\eta$-inverse strongly monotone} if there is a positive number $ \eta $ such that
\[
\langle \mathcal{A} p-\mathcal{A} q, p-q\rangle \geq \eta\|\mathcal{A} p-\mathcal{A} q\|^{2},
\]
\item \emph{monotone} if
\[
\langle \mathcal{A} p-\mathcal{A} q, p-q\rangle \geq 0,
\]
\item \emph{$\eta$-strongly pseudomonotone} if there is a positive number $ \eta $ such that
\[
\langle \mathcal{A} p, q-p\rangle \geq 0 \Longrightarrow\langle \mathcal{A} q, q-p\rangle \geq \eta\|p-q\|^{2},
\]
\item \emph{pseudomonotone} if
\[
\langle \mathcal{A} p, q-p\rangle \geq 0 \Longrightarrow\langle \mathcal{A} q, q-p\rangle \geq 0,
\]
\item \emph{$ L $-Lipschitz continuous} if there is $ L >0$ such that
\[
\|\mathcal{A} p-\mathcal{A} q\| \leq L\|p-q\|,
\]
\item \emph{sequentially weakly
continuous} if for any sequence $ \{p_{n}\} $ weakly converges to a point $ p \in {H} $,  $ \left\{\mathcal{A}p_{n}\right\} $ weakly converges to $ \mathcal{A}p $.
\end{enumerate}
It can be easily checked that the following  relations: $(1)\Longrightarrow(3) \Longrightarrow(5)$ and $(1)\Longrightarrow(4) \Longrightarrow(5)$. Note that the opposite statement is generally incorrect. Recall that a mapping ${P}_{{C}}: {H} \rightarrow {C}$ is called the metric projection from ${H}$ onto $C$, if for all $x \in H$, there is a unique nearest point in $ C $, which is represented by $P_{C}(x)$, such that $P_{C}x:= \operatorname{argmin}\{\|x-y\|,\, y \in C\}$.

The oldest and simplest projection approach to solve variational inequality problems is the projected-gradient method, which  reads as follows:
\begin{equation*}\label{PGM}
x_{n+1}=P_{C}\left(x_{n}-\gamma \mathcal{A} x_{n}\right), \quad \forall n \geq 1\,,
\tag{\text{PGM}}
\end{equation*}
where $ P_{C} $ represents the metric projection   onto $ C $, mapping $ \mathcal{A}$  is $ L $-Lipschitz continuous and $ \eta $-strongly monotone and the step size $ \gamma \in(0, {2 \eta}/{L^{2}}) $. Then the iterative sequence $ \{x_{n}\} $ defined by~\eqref{PGM} converges to the solution of~\eqref{VIP} provided that $ \mathrm{VI}(C,\mathcal{A}) $ is nonempty. It should be noted that the iterative sequence $ \{x_{n}\} $ formulated by~\eqref{PGM} does not necessarily converge when mapping $ \mathcal{A} $ is ``only" monotone.  Recently, Malitsky~\cite{PRGM} introduced a projected reflected gradient method, which can be viewed as an improvement of \eqref{PGM}. Indeed, the sequence generated by this method is as follows:
\begin{equation*}\label{PRGM}
x_{n+1}=P_{C}\left(x_{n}-\gamma \mathcal{A}\left(2 x_{n}-x_{n-1}\right)\right), \quad \forall n \geq 1\,.
\tag{\text{PRGM}}
\end{equation*}
He proved that the sequence $ \{x_{n}\} $ created by iterative scheme \eqref{PRGM} converges to $ u \in \mathrm{VI}(C,\mathcal{A}) $ when the mapping $ \mathcal{A} $ is monotone. Further extensions of \eqref{PRGM} can be found in \cite{PRGM1,PRGM2}.

In many kinds of research on solving variational inequalities controlled by pseudomonotone and Lipschitz continuous operators, the most commonly used algorithm is the extragradient method (see~\cite{EGM}) and its variants. Indeed, Korpelevich proposed the extragradient method (EGM) in~\cite{EGM} to find the solution of the saddle point problem in finite-dimensional spaces. The details of EGM are described as follows:
\begin{equation*}\label{EGM}
\left\{\begin{aligned}
&y_{n}=P_{C}\left(x_{n}-\gamma \mathcal{A} x_{n}\right), \\
&x_{n+1}=P_{C}\left(x_{n}-\gamma \mathcal{A} y_{n}\right), \quad \forall n \geq 1\,,
\end{aligned}\right.
\tag{\text{EGM}}
\end{equation*}
where mapping $\mathcal{A}$ is $ L $-Lipschitz continuous monotone and fixed step size $\gamma \in(0, 1/L)$. Under the condition of $ \mathrm{VI}(C, \mathcal{A}) \ne \emptyset$, the iterative sequence $ \{x_{n}\} $ defined by \eqref{EGM} converges to an element of $ \mathrm{VI}(C, \mathcal{A}) $. In the past few decades, EGM has been considered and extended by many authors for solving \eqref{VIP} in infinite-dimensional spaces, see, e.g., \cite{SILD,SLD,tanarxiv} and the references therein. Recently, Vuong~\cite{EGMpVIP} extended EGM to solve pseudomonotone variational inequalities in Hilbert spaces, and proved that the iterative sequence constructed by the algorithm converges weakly to a solution of~\eqref{VIP}. On the other hand, it is not easy to calculate the projection on the general closed convex set $ C $, especially when $ C $ has a complex structure. Note that in the extragradient method, two projections need to be calculated on the closed convex set $ C $ for each iteration, which may severely affect the computational performance of the algorithm used.

Next, we introduce two types of methods to enhance the numerical efficiency of EGM. The first approach is the Tseng's extragradient method (referred to as TEGM, also known as the forward-backward-forward method) proposed by Tseng~\cite{tseng}. The advantage of this method is that the projection on the feasible set only needs to be calculated once in each iteration. More precisely, TEGM is expressed as follows:
\begin{equation*}\label{TEGM}
\left\{\begin{aligned}
&y_{n}=P_{C}\left(x_{n}-\gamma \mathcal{A} x_{n}\right)\,, \\
&x_{n+1}=y_{n}-\gamma\left(\mathcal{A} y_{n}-\mathcal{A} x_{n}\right), \quad \forall n \geq 1\,,
\end{aligned}\right.
\tag{\text{TEGM}}
\end{equation*}
where mapping $\mathcal{A}$ is $ L $-Lipschitz continuous monotone and fixed step size $\gamma \in(0, 1/L)$. Then the iterative sequence $ \{x_{n}\} $ formulated by \eqref{TEGM} converges to a solution of \eqref{VIP} provided that $ \mathrm{VI}(C, \mathcal{A}) $ is nonempty. Very recently, Bot, Csetnek and Vuong in their recent work~\cite{BotpVIP} proposed a Tseng's forward-backward-forward algorithm for solving pseudomonotone variational inequalities in Hilbert spaces and performed an asymptotic analysis of the formed trajectories. The second method is the subgradient extragradient method (SEGM) proposed by Censor, Gibali and Reich~\cite{SEGM}. This can be regarded as a modification of EGM. Indeed, they  replaced the second projection in \eqref{EGM} by a projection onto a half-space. SEGM is calculated as follows:
\begin{equation*}\label{SEGM}
\left\{\begin{aligned}
&y_{n}=P_{C}\left(x_{n}-\gamma \mathcal{A} x_{n}\right)\,, \\
&T_{n}=\left\{x \in H \mid \langle x_{n}-\gamma \mathcal{A} x_{n}-y_{n}, x-y_{n}\rangle \leq 0\right\}\,, \\
&x_{n+1}=P_{T_{n}}\left(x_{n}-\gamma \mathcal{A} y_{n}\right), \quad \forall n \geq 1\,,
\end{aligned}\right.
\tag{\text{SEGM}}
\end{equation*}
where mapping $\mathcal{A}$ is $ L $-Lipschitz continuous monotone and fixed step size $\gamma \in(0, 1/L)$. SEGM not only converges to monotone variational inequalities (see~\cite{CGR1}), but also to pseudomonotone variational inequalities (see~\cite{CGR2,TSI}).

It is worth mentioning that   \eqref{EGM}, \eqref{TEGM} and \eqref{SEGM} are   weakly convergent in infinite-dimensional Hilbert spaces. Some practical problems that occur in the fields of image processing, quantum mechanics, medical imaging and machine learning need to be modeled and analyzed in infinite-dimensional space. Therefore, strong convergence results are preferable to weak convergence results in infinite-dimensional space. Recently, Thong and Vuong~\cite{MaTEGM} introduced the modified Mann-type Tseng's extragradient method to solve the \eqref{VIP} involving a pseudomonotone mapping in Hilbert spaces. Their method uses an Armijo-like line search to eliminate the reliance on the Lipschitz continuous constant of the mapping $ \mathcal{A} $. Indeed, the proposed algorithm is stated as follows:
\begin{equation*}\label{MaTEGM}
\left\{\begin{aligned}
&y_{n}=P_{C}\left(x_{n}-\gamma_{n} \mathcal{A} x_{n}\right)\,,\\
&z_{n}=y_{n}-\gamma_{n}\left(\mathcal{A} y_{n}-\mathcal{A} x_{n}\right)\,,\\
&x_{n+1}=\left(1-\varphi_{n}-\tau_{n}\right) x_{n}+\tau_{n} z_{n}, \quad \forall n \geq 1\,,
\end{aligned}\right.
\tag{\text{MaTEGM}}
\end{equation*}
where the mapping $ \mathcal{A} $ is pseudomonotone, sequentially weakly continuous on $ C $ and uniformly continuous on bounded subsets of $ H $, $\left\{\varphi_{n}\right\}$, $\left\{\tau_{n}\right\}$ are two real positive sequences in $ (0,1) $ such that $\left\{\tau_{n}\right\} \subset \left(a, 1-\varphi_{n}\right)$ for some $a>0$ and $\lim _{n \rightarrow \infty} \varphi_{n}=0$, $ \sum_{n=1}^{\infty} \varphi_{n}=\infty$, $\gamma_{n}:=\alpha \ell^{q_{n}}$ and $q_{n}$ is the smallest non-negative integer $ q $ satisfying $\alpha \ell^{q}\left\|\mathcal{A} x_{n}-\mathcal{A} y_{n}\right\| \leq \phi\left\|x_{n}-y_{n}\right\|$ ($ \alpha>0$, $\ell \in(0,1)$, $\phi \in(0,1) $). They showed that the iteration scheme formed by \eqref{MaTEGM} converges strongly to an element $ u $ under $ \mathrm{VI}(C,\mathcal{A}) \ne \emptyset $, where $ u =\arg \min \{\|z\|: z \in \mathrm{VI}(C,\mathcal{A})\} $.

To accelerate the convergence rate of the algorithms, in 1964, Polyak \cite{inertial} considered the second-order dynamical system $\ddot{x}(t)+\gamma \dot{x}(t)+\nabla f(x(t))=0$, where $\gamma>0$, $\nabla f$ represents the gradient of $ f $, $ \dot{x}(t) $ and $ \ddot{x}(t) $ denote the first and second derivatives of $ x $ at $ t $, respectively. This dynamic system is called the Heavy Ball with Friction (HBF).

Next, we consider the discretization of this dynamic system (HBF), that is,
\[
\frac{x_{n+1}-2 x_{n}+x_{n-1}}{h^{2}}+\gamma \frac{x_{n}-x_{n-1}}{h}+\nabla f\left(x_{n}\right)=0, \quad \forall n \geq 0\,.
\]
Through a direct calculation, we can get the following form:
\[
x_{n+1}=x_{n}+\tau\left(x_{n}-x_{n-1}\right)-\varphi \nabla f\left(x_{n}\right), \quad \forall n \geq 0\,,
\]
where $\tau=1-\gamma h$ and $\varphi=h^{2}$. This can be considered as the following two-step iteration scheme:
\[
\left\{\begin{aligned}
&y_{n}=x_{n}+\tau\left(x_{n}-x_{n-1}\right)\,, \\
&x_{n+1}=y_{n}-\varphi \nabla f\left(x_{n}\right), \quad \forall n \geq 0\,.
\end{aligned}\right.
\]
This iteration is now called the inertial extrapolation algorithm, the term $\tau \left(x_{n}-x_{n-1}\right)$ is referred to as the extrapolation point. In recent years,  inertial technology as an acceleration method has attracted extensive research in the optimization community. Many scholars have built various fast numerical algorithms by employing the inertial technology.  These algorithms have shown advantages in theoretical and computational experiments and have been successfully applied to many problems, see, for instance, \cite{FISTA,GHjfpta,zhoucoam} and the references therein.

Very recently, inspired by the inertial method, the SEGM and the viscosity  method, Thong, Hieu and Rassias~\cite{ViSEGM} presented a viscosity-type inertial subgradient extragradient algorithm to solve pseudomonotone \eqref{VIP} in Hilbert spaces. The algorithm is of the form:
\begin{equation}\label{ViSEGM}
\left\{\begin{aligned}
&s_{n}=x_{n}+\delta_{n}\left(x_{n}-x_{n-1}\right)\,, \\
&y_{n}=P_{C}\left(s_{n}-\gamma_{n} \mathcal{A} s_{n}\right) \,,\\
&T_{n}=\left\{x \in H \mid \langle s_{n}-\gamma_{n} \mathcal{A} s_{n}-y_{n}, x-y_{n}\rangle \leq 0\right\}\,, \\
&z_{n}=P_{T_{n}}\left(s_{n}-\gamma_{n} \mathcal{A} y_{n}\right)\,,\\
&x_{n+1}=\varphi_{n} f\left(z_{n}\right)+\left(1-\varphi_{n}\right) z_{n}, \quad \forall n\geq 1\,.\\
&\gamma_{n+1}=\left\{\begin{array}{ll}
\min \left\{\frac{\phi\left\|s_{n}-y_{n}\right\|}{\left\|\mathcal{A} s_{n}-\mathcal{A} y_{n}\right\|}, \gamma_{n}\right\}, & \text { if } \mathcal{A} s_{n}-\mathcal{A} y_{n} \neq 0; \\
\gamma_{n}, & \text { otherwise},
\end{array}\right.
\end{aligned}\right.
\tag{\text{ViSEGM}}
\end{equation}
where the mapping $ \mathcal{A} $ is  pseudomonotone, $ L $-Lipschitz continuous, sequentially weakly continuous on $ C $, and the inertia parameters $ \delta_{n} $ are updated in the following ways:
\[
\delta_{n}=\left\{\begin{array}{ll}
\min \bigg\{\dfrac{\epsilon_{n}}{\left\|x_{n}-x_{n-1}\right\|}, \delta\bigg\}, & \text { if } x_{n} \neq x_{n-1}; \\
\delta, & \text { otherwise}.
\end{array}\right.
\]
Note that the Algorithm~\eqref{ViSEGM} uses a simple step size rule, which is generated through some calculations  of previously known information in each iteration. Therefore, it can work well without the prior information of the Lipschitz constant of the mapping $ \mathcal{A} $. They confirmed the strong convergence of~\eqref{ViSEGM} under mild assumptions on cost mapping and parameters.

Motivated and stimulated by the above works,  we introduce a new inertial Tseng's extragradient algorithm with a new step size for solving the pseudomonotone \eqref{VIP} in Hilbert spaces. The advantages of our algorithm are: (1) only one projection on the feasible set needs to be calculated in each iteration; (2) do not require to know the prior information of the Lipschitz constant of the cost mapping; (3) the addition of inertial makes it have faster convergence speed. Under mild assumptions, we confirm a strong convergence theorem of the suggested algorithm. Lastly, some computational tests appearing in finite and infinite dimensions are proposed to verify our theoretical results. Furthermore, our algorithm is also designed to solve optimal control problems. Our algorithm improves some existing results \cite{MaTEGM,ViSEGM,THna,YLNA,FQOPT2020}.

The organizational structure of our paper is built up as follows. Some essential definitions and technical lemmas that need to be used are given in the next section. In Section~\ref{sec3}, we propose an algorithm and analyze its convergence. Some computational tests and applications to verify our theoretical results are presented in Section~\ref{sec4}. Finally, the paper ends with a brief summary.
\section{Preliminaries}\label{sec2}
Let $ C $ be a  closed and convex nonempty subset of a real Hilbert space $ H $. The weak convergence and strong convergence of $\left\{x_{n}\right\}_{n=1}^{\infty}$ to $x$ are represented by $x_{n} \rightharpoonup x$ and $x_{n} \rightarrow x$, respectively. For each $x, y \in H$ and $\delta \in \mathbb{R}$, we have the following facts:
\begin{enumerate}[label=(\arabic*)]
\item $\|x+y\|^{2} \leq\|x\|^{2}+2\langle y, x+y\rangle$;
\item $\|\delta x+(1-\delta) y\|^{2}=\delta\|x\|^{2}+(1-\delta)\|y\|^{2}-\delta(1-\delta)\|x-y\|^{2}$.
\end{enumerate}

It is known that $P_{C} x$ has the following basic properties:
\begin{itemize}[leftmargin=1em]
\item $ \langle x-P_{C} x, y-P_{C} x\rangle \leq 0, \, \forall y \in C $;
\item $\left\|P_{C} x-P_{C} y\right\|^{2} \leq\langle P_{C} x-P_{C} y, x-y\rangle, \,\forall y \in H$;
\item $\left\|x-P_{C}(x)\right\|^{2} \leq\|x-y\|^{2}-\left\|y-P_{C}(x)\right\|^{2}, \, \forall y \in C$.
\end{itemize}

We give some explicit formulas to calculate projections on special feasible sets.
\begin{enumerate}[label=(\arabic*)]
\item The projection of $ x $ onto a half-space $H_{u, v}=\{x:\langle u, x\rangle \leq v\}$ is given by
\[P_{H_{u, v}}(x)=x-\max\{{[\langle u, x\rangle-v]}/{\|u\|^{2}},0\} u\,.
\]
\item The projection of $ x $ onto a box $\operatorname{Box}[a, b]=\{x: a \leq x \leq b\}$ is given by
\[
P_{\mathrm{Box}[a, b]}(x)_{i}=\min \left\{ b_{i},\max \left\{x_{i}, a_{i}\right\}\right\}\,.
\]
\item The projection of $ x $ onto a ball $B[p, q]=\{x:\|x-p\| \leq q\}$ is given by
\[
P_{B[p, q]}(x)=p+\frac{q}{\max \{\|x-p\|, q\}}(x-p)\,.
\]
\end{enumerate}

The following lemmas play an important role in our proof.
\begin{lemma}[\cite{CY1992}]\label{lem21}
Assume that $ C $ is a closed and convex subset of a real Hilbert space $H$. Let operator ${\mathcal{A}}: {C} \rightarrow {H}$ be continuous and pseudomonotone. Then, $x^{*}$ is a solution of \eqref{VIP} if and only if $ \langle \mathcal{A} x, x-x^{*}\rangle \geq 0,\, \forall x \in C $.
\end{lemma}
\begin{lemma}[\cite{SY2012}]\label{lem22}
Let $\left\{p_{n}\right\}$ be a positive sequence, $\left\{q_{n}\right\}$ be a sequence of real numbers, and $\left\{\sigma_{n}\right\}$ be a sequence in $ (0,1) $ such that $\sum_{n=1}^{\infty} \sigma_{n}=\infty$. Assume that
\[
p_{n+1} \leq\left(1-\sigma_{n}\right) p_{n}+\sigma_{n} q_{n}, \quad \forall n \geq 1\,.
\]
If $\limsup _{k \rightarrow \infty} q_{n_{k}} \leq 0$ for every subsequence $\left\{p_{n_{k}}\right\}$ of $\left\{p_{n}\right\}$ satisfying $\lim \inf _{k \rightarrow \infty}$
$\left(p_{n_{k}+1}-p_{n_{k}}\right) \geq~0$, then $\lim _{n \rightarrow \infty} p_{n}=0$.
\end{lemma}
\section{Main results}\label{sec3}
In this section, we present a self adaptive inertial viscosity-type Tseng's extragradient algorithm, which is based on the inertial method, the viscosity method and the Tseng's extragradient method. The major benefit of this algorithm is that the step size is automatically updated at each iteration without performing any line search procedure. Moreover, our iterative scheme only needs to calculate the projection once in each iteration. Before starting to state our main result, we assume that our algorithm satisfies the following five assumptions.
\begin{enumerate}[label=(C\arabic*)]
\item The feasible set $ C $ is  closed, convex and nonempty. \label{con1}
\item The solution set of the \eqref{VIP} is nonempty, that is, $\mathrm{VI}(C,\mathcal{A}) \neq \emptyset$.\label{con2}
\item The mapping $\mathcal{A}: H \rightarrow H$ is pseudomonotone and $ L $-Lipschitz continuous on $H$, and sequentially weakly continuous on $C$.  \label{con3}
\item The mapping $f: H \rightarrow H$ is $\rho$-contractive with $\rho \in[0,1)$.\label{con4}
\item The positive sequence $ \{\epsilon_{n}\} $ satisfies $\lim_{n \rightarrow \infty} \frac{\epsilon_{n}}{\varphi_{n}}=0$, where $ \{\varphi_{n}\}\subset (0,1) $ such that
$\lim _{n \rightarrow \infty} \varphi_{n}=0$ and $\sum_{n=1}^{\infty} \varphi_{n}=\infty$.  \label{con5}
\end{enumerate}

Now, we can state the details of the iterative method. Our algorithm is described as follows.
\begin{algorithm}[h]
\caption{Self adaptive inertial viscosity-type Tseng's extragradient algorithm}
\label{alg1}
\begin{algorithmic}
\STATE {\textbf{Initialization:} Given $ \delta>0 $, $\gamma_{1}>0$, $\phi \in(0,1)$. Let $x_{0},x_{1} \in H$  be two initial points.}
\STATE \textbf{Iterative Steps}: Calculate the next iteration point $ x_{n+1} $ as follows:
\begin{equation*}
\left\{\begin{aligned}
&s_{n}=x_{n}+\delta_{n}\left(x_{n}-x_{n-1}\right) \,,\\
&y_{n}=P_{C}\left(s_{n}-\gamma_{n} \mathcal{A} s_{n}\right) \,,\\
&z_{n}=y_{n}-\gamma_{n}\left(\mathcal{A} y_{n}-\mathcal{A} s_{n}\right)\,,\\
&x_{n+1}=\varphi_{n} f\left(z_{n}\right) + \left(1-\varphi_{n}\right) z_{n}\,.
\end{aligned}\right.
\end{equation*}
where $ \{\delta_{n}\} $ and $ \{\gamma_{n}\} $ are updated by \eqref{alpha} and \eqref{lambda}, respectively.
\begin{equation}\label{alpha}
\delta_{n}=\left\{\begin{array}{ll}
\min \bigg\{\dfrac{\epsilon_{n}}{\|x_{n}-x_{n-1}\|}, \delta\bigg\}, & \text { if } x_{n} \neq x_{n-1}; \\
\delta, & \text { otherwise}.
\end{array}\right.
\end{equation}
\begin{equation}\label{lambda}
\gamma_{n+1}=\left\{\begin{array}{ll}
\min \left\{\dfrac{\phi\|s_{n}-y_{n}\|}{\|\mathcal{A} s_{n}-\mathcal{A} y_{n}\|}, \gamma_{n}\right\}, & \text { if } \mathcal{A} s_{n}-\mathcal{A} y_{n} \neq 0; \\
\gamma_{n}, & \text { otherwise}.
\end{array}\right.
\end{equation}
\end{algorithmic}
\end{algorithm}

\begin{remark}\label{rem31}
It follows from \eqref{alpha} and Assumption~\ref{con5} that
\[
\lim _{n \rightarrow \infty} \frac{\delta_{n}}{\varphi_{n}}\left\|x_{n}-x_{n-1}\right\|=0\,.
\]
Indeed, we obtain $\delta_{n}\left\|x_{n}-x_{n-1}\right\| \leq \epsilon_{n}, \forall n$, which together with $\lim _{n \rightarrow \infty} \frac{\epsilon_{n}}{\varphi_{n}}=0$ yields
\[
\lim _{n \rightarrow \infty} \frac{\delta_{n}}{\varphi_{n}}\left\|x_{n}-x_{n-1}\right\| \leq \lim _{n \rightarrow \infty} \frac{\epsilon_{n}}{\varphi_{n}}=0\,.
\]
\end{remark}
\begin{lemma}\label{lem31}
The sequence $\left\{\gamma_{n}\right\}$ formulated by \eqref{lambda} is  nonincreasing and
\[
\lim _{n \rightarrow \infty} \gamma_{n}=\gamma \geq \min \Big\{\gamma_{1}, \frac{\phi}{L}\Big\}\,.
\]
\end{lemma}
\begin{proof}
On account of \eqref{lambda}, we have $\gamma_{n+1} \leq \gamma_{n}, \forall n \in \mathbb{N} $. Hence, $\left\{\gamma_{n}\right\}$ is nonincreasing. Moreover, we get that $\left\|\mathcal{A} s_{n}-\mathcal{A} y_{n}\right\| \leq L\left\|s_{n}-y_{n}\right\|$ by means of $\mathcal{A}$ is $L$-Lipschitz continuous. Thus,
\[
\phi \frac{\left\|s_{n}-y_{n}\right\|}{\left\|\mathcal{A} s_{n}-\mathcal{A} y_{n}\right\|} \geq \frac{\phi}{L}, \,\,\text {  if  }\,\, \mathcal{A} s_{n} \neq \mathcal{A} y_{n}\,,
\]
which together with \eqref{lambda} implies that $ \gamma_{n} \geq \min \{\gamma_{1}, \frac{\phi}{L}\} $. Therefore,  $\lim _{n \rightarrow \infty} \gamma_{n}=\gamma \geq \min \big\{\gamma_{1}, \frac{\phi}{L}\big\}$ since sequence $ \{\gamma_{n}\} $  is lower bounded and nonincreasing. 
\end{proof}
The following lemmas have a significant part to play in the convergence proof of our algorithm.
\begin{lemma}\label{lem32}
Suppose that Assumptions \ref{con1}--\ref{con3} hold. Let $\{s_{n}\}$ and $ \{y_{n} \}$ be two sequences formulated by Algorithm~\ref{alg1}. If there exists a subsequence $\{s_{n_{k}}\}$ convergent weakly to $z \in H$ and $\lim _{k \rightarrow \infty}\|s_{n_{k}}-y_{n_{k}}\|=0$, then $z \in \mathrm{VI}(C, \mathcal{A})$.
\end{lemma}
\begin{proof}
From the property of projection and $ y_{n}=P_{C}\left(s_{n}-\gamma_{n} \mathcal{A} s_{n}\right) $, we have
\[
\langle s_{n_{k}}-\gamma_{n_{k}} \mathcal{A} s_{n_{k}}-y_{n_{k}}, x-y_{n_{k}}\rangle \leq 0, \quad \forall x \in C\,,
\]
which can be written as follows
\[
\frac{1}{\gamma_{n_{k}}}\langle s_{n_{k}}-y_{n_{k}}, x-y_{n_{k}}\rangle \leq\langle \mathcal{A} s_{n_{k}}, x-y_{n_{k}}\rangle, \quad \forall x \in C\,.
\]	
Through a direct calculation, we get
\begin{equation}\label{aw}
\frac{1}{\gamma_{n_{k}}}\langle s_{n_{k}}-y_{n_{k}}, x-y_{n_{k}}\rangle+\langle \mathcal{A} s_{n_{k}}, y_{n_{k}}-s_{n_{k}}\rangle \leq\langle \mathcal{A} s_{n_{k}}, x-s_{n_{k}}\rangle,\quad  \forall x \in C\,.
\end{equation}
We have that $\{s_{n_{k}}\}$ is bounded since $\{s_{n_{k}}\}$ is convergent weakly to $z \in H$. Then, from the Lipschitz continuity of $\mathcal{A}$ and $\|s_{n_{k}}-y_{n_{k}}\| \rightarrow 0$, we obtain that  $\{\mathcal{A} s_{n_{k}}\}$ and $\{y_{n_{k}}\}$ are also bounded. Since $\gamma_{n_{k}} \geq \min \{\gamma_{1}, \frac{\phi}{L}\}$, one concludes from \eqref{aw} that
\begin{equation}\label{po}
\liminf _{k \rightarrow \infty}\langle \mathcal{A} s_{n_{k}}, x-s_{n_{k}}\rangle \geq 0, \quad \forall x \in C\,.
\end{equation}
Moreover, one has
\begin{equation}\label{pi}
\begin{aligned}
\langle \mathcal{A} y_{n_{k}}, x-y_{n_{k}}\rangle=&\langle \mathcal{A} y_{n_{k}}-\mathcal{A} s_{n_{k}}, x-s_{n_{k}}\rangle +\langle \mathcal{A} s_{n_{k}}, x-s_{n_{k}}\rangle+\langle \mathcal{A} y_{n_{k}}, s_{n_{k}}-y_{n_{k}}\rangle\,.
\end{aligned}
\end{equation}
Since $\lim _{k \rightarrow \infty}\|s_{n_{k}}-y_{n_{k}}\|=0$ and $\mathcal{A}$ is Lipschitz continuous, we get $ \lim _{k \rightarrow \infty}\|\mathcal{A} s_{n_{k}}-\mathcal{A} y_{n_{k}}\|=0 $.
This together with \eqref{po} and \eqref{pi} yields that $ \liminf _{k \rightarrow \infty}\langle \mathcal{A} y_{n_{k}}, x-y_{n_{k}}\rangle \geq 0 $.

Next, we select a positive number decreasing sequence $\{\zeta_{k}\}$ such that $ \zeta_{k}\to 0 $ as $ k \to \infty $. For any $k$, we represent the smallest positive integer with $N_{k}$ such that
\begin{equation}\label{pp}
\langle \mathcal{A} y_{n_{j}}, x-y_{n_{j}}\rangle+\zeta_{k} \geq 0,\quad  \forall j \geq N_{k}\,.
\end{equation}
It can be easily seen that the sequence $\{N_{k}\}$ is increasing because   $\{\zeta_{k}\}$ is decreasing. Moreover, for any $k$, from $\{y_{N_{k}}\} \subset C$, we can assume $\mathcal{A} y_{N_{k}} \neq 0$ (otherwise, $y_{N_{k}}$ is a solution) and set  $ u_{N_{k}}={\mathcal{A} y_{N_{k}}}/{\|\mathcal{A} y_{N_{k}}\|^{2}} $. Then, we get $\langle \mathcal{A} y_{N_{k}}, u_{N_{k}}\rangle=1, \forall k$. Now, we can deduce from \eqref{pp} that $ \langle \mathcal{A} y_{N_{k}}, x+\zeta_{k} u_{N_{k}}-y_{N_{k}}\rangle \geq 0,\forall k $. According to the fact that  $\mathcal{A}$ is pseudomonotone on $H$, we can show that
\[
\langle \mathcal{A}\left(x+\zeta_{k} u_{N_{k}}\right), x+\zeta_{k} u_{N_{k}}-y_{N_{k}}\rangle \geq 0\,,
\]
which further yields that
\begin{equation}\label{pu}
\langle \mathcal{A} x, x-y_{N_{k}}\rangle \geq\langle \mathcal{A} x-\mathcal{A}\left(x+\zeta_{k} u_{N_{k}}\right), x+\zeta_{k} u_{N_{k}}-y_{N_{k}}\rangle-\zeta_{k}\langle \mathcal{A} x, u_{N_{k}}\rangle\,.
\end{equation}
Now, we prove that $\lim _{k \rightarrow \infty} \zeta_{k} u_{N_{k}}=0 $. We get that $y_{N_{k}} \rightharpoonup z$ since $s_{n_{k}} \rightharpoonup z$ and $\lim _{k \rightarrow \infty} \| s_{n_{k}}-$
$y_{n_{k}} \|=0$. From $\{y_{n}\} \subset C$, we have $z \in C $.  In view of $\mathcal{A}$ is sequentially weakly continuous on~$C$, one has that $\{\mathcal{A} y_{n_{k}}\}$ converges weakly to $\mathcal{A} z $. One assumes that $\mathcal{A} z \ne 0$ (otherwise, $z$ is a solution). According to the fact that norm mapping is sequentially weakly lower semicontinuous, we obtain $ 0<\|\mathcal{A} z\| \leq \liminf _{k \rightarrow \infty}\|\mathcal{A} y_{n_{k}}\| $. Using $\{y_{N_{k}}\} \subset\{y_{n_{k}}\}$ and $\zeta_{k} \rightarrow 0$ as $k \rightarrow \infty$, we have
\[
0 \leq \limsup _{k \rightarrow \infty}\|\zeta_{k} u_{N_{k}}\|=\limsup _{k \rightarrow \infty}\Big(\frac{\zeta_{k}}{\|\mathcal{A} y_{n_{k}}\|}\Big) \leq \frac{\lim \sup _{k \rightarrow \infty} \zeta_{k}}{\lim \inf _{k \rightarrow \infty}\|\mathcal{A} y_{n_{k}}\|}=0\,.
\]
That is, $\lim _{k \rightarrow \infty} \zeta_{k} u_{N_{k}}=0$. Thus, from the facts that  $\mathcal{A}$ is Lipschitz continuous, sequences $\{y_{N_{k}}\}$ and $\{u_{N_{k}}\}$ are bounded and $\lim _{k \rightarrow \infty} \zeta_{k} u_{N_{k}}=0 $, we can conclude from \eqref{pu} that $ \liminf _{k \rightarrow \infty}\langle \mathcal{A} x, x-y_{N_{k}}\rangle \geq 0 $.
Therefore,
\[
\langle \mathcal{A} x, x-z\rangle=\lim _{k \rightarrow \infty}\langle \mathcal{A} x, x-y_{N_{k}}\rangle=\liminf _{k \rightarrow \infty}\langle \mathcal{A} x, x-y_{N_{k}}\rangle \geq 0, \forall x\in C\,.
\]
Consequently, we observe that $z \in \mathrm{VI}({C}, {\mathcal{A}})$ by  Lemma~\ref{lem21}.  This completes the proof. 
\end{proof}
\begin{remark}
If $\mathcal{A}$ is monotone,  then $\mathcal{A}$ does not need to satisfy sequential weak continuity, see \cite{DSC}.
\end{remark}

\begin{lemma}\label{lem33}
Suppose that Assumptions \ref{con1}--\ref{con3} hold. Let sequences $\{z_{n}\}$ and $ \{y_{n}\} $ be formulated by Algorithm~\ref{alg1}. Then, we have
\[
\|z_{n}-u\|^{2} \leq\|s_{n}-u\|^{2}-\Big(1-\phi^{2} \frac{\gamma_{n}^{2}}{\gamma_{n+1}^{2}}\Big)\|s_{n}-y_{n}\|^{2},\quad  \forall u \in \mathrm{VI}(C, \mathcal{A})\,,
\]
and
\[
\|z_{n}-y_{n}\| \leq \phi \frac{\gamma_{n}}{\gamma_{n+1}}\|s_{n}-y_{n}\|\,.
\]
\end{lemma}
\begin{proof}
First, using the definition of $\left\{\gamma_{n}\right\}$, one obtains
\begin{equation}\label{q}
\left\|\mathcal{A} s_{n}-\mathcal{A} y_{n}\right\| \leq \frac{\phi}{\gamma_{n+1}}\left\|s_{n}-y_{n}\right\|, \quad \forall n\,.
\end{equation}
Indeed, if $\mathcal{A} s_{n}=\mathcal{A} y_{n}$ then \eqref{q} clearly holds. Otherwise, it follows from \eqref{lambda} that
\[
\gamma_{n+1}=\min \left\{\frac{\phi\left\|s_{n}-y_{n}\right\|}{\left\|\mathcal{A} s_{n}-\mathcal{A} y_{n}\right\|}, \gamma_{n}\right\} \leq \frac{\phi\left\|s_{n}-y_{n}\right\|}{\left\|\mathcal{A} s_{n}-\mathcal{A} y_{n}\right\|}\,.
\]
Consequently, we have
\[
\left\|\mathcal{A} s_{n}-\mathcal{A} y_{n}\right\| \leq \frac{\phi}{\gamma_{n+1}}\left\|s_{n}-y_{n}\right\|\,.
\]
Therefore,   inequality \eqref{q} holds when $\mathcal{A} s_{n}=\mathcal{A} y_{n}$ and $\mathcal{A} s_{n} \neq \mathcal{A} y_{n}$. From the definition of $ z_{n} $, one sees that
\begin{equation}\label{a}
\begin{aligned}
\|z_{n}-u\|^{2}=&\|y_{n}-\gamma_{n}\left(\mathcal{A} y_{n}-\mathcal{A} s_{n}\right)-u\|^{2} \\
=&\|y_{n}-u\|^{2}+\gamma_{n}^{2}\|\mathcal{A} y_{n}-\mathcal{A} s_{n}\|^{2}-2 \gamma_{n}\langle y_{n}-u, \mathcal{A} y_{n}-\mathcal{A} s_{n}\rangle \\
=&\|s_{n}-u\|^{2}+\|y_{n}-s_{n}\|^{2}+2\langle y_{n}-s_{n}, s_{n}-u\rangle \\
&+\gamma_{n}^{2}\|\mathcal{A} y_{n}-\mathcal{A} s_{n}\|^{2}-2 \gamma_{n}\langle y_{n}-u, \mathcal{A} y_{n}-\mathcal{A} s_{n}\rangle \\
=&\|s_{n}-u\|^{2}+\|y_{n}-s_{n}\|^{2}-2\langle y_{n}-s_{n}, y_{n}-s_{n}\rangle+2\langle y_{n}-s_{n}, y_{n}-u\rangle \\
&+\gamma_{n}^{2}\|\mathcal{A} y_{n}-\mathcal{A} s_{n}\|^{2}-2 \gamma_{n}\langle y_{n}-u, \mathcal{A} y_{n}-\mathcal{A} s_{n}\rangle \\
=&\|s_{n}-u\|^{2}-\|y_{n}-s_{n}\|^{2}+2\langle y_{n}-s_{n}, y_{n}-u\rangle \\
&+\gamma_{n}^{2}\|\mathcal{A} y_{n}-\mathcal{A} s_{n}\|^{2}-2 \gamma_{n}\langle y_{n}-u, \mathcal{A} y_{n}-\mathcal{A} s_{n}\rangle\,.
\end{aligned}
\end{equation}
Since $y_{n}=P_{C}\left(s_{n}-\gamma_{n} \mathcal{A} s_{n}\right)$, using the property of projection, we obtain
\[
\langle y_{n}-s_{n}+\gamma_{n} \mathcal{A} s_{n}, y_{n}-u\rangle \leq 0\,,
\]
or equivalently
\begin{equation}\label{z}
\langle y_{n}-s_{n}, y_{n}-u\rangle \leq-\gamma_{n}\langle \mathcal{A} s_{n}, y_{n}-u\rangle\,.
\end{equation}
From \eqref{q}, \eqref{a} and \eqref{z}, we have
\begin{equation}\label{w}
\begin{aligned}
\|z_{n}-u\|^{2}  \leq&\|s_{n}-u\|^{2}-\|y_{n}-s_{n}\|^{2}-2 \gamma_{n}\langle \mathcal{A} s_{n}, y_{n}-u\rangle+\phi^{2} \frac{\gamma_{n}^{2}}{\gamma_{n+1}^{2}}\|s_{n}-y_{n}\|^{2} \\
&-2 \gamma_{n}\langle y_{n}-u, \mathcal{A} y_{n}-\mathcal{A} s_{n}\rangle \\
\leq&\|s_{n}-u\|^{2}-\Big(1-\phi^{2} \frac{\gamma_{n}^{2}}{\gamma_{n+1}^{2}}\Big)\|s_{n}-y_{n}\|^{2}-2 \gamma_{n}\langle y_{n}-u, \mathcal{A} y_{n}\rangle\,.
\end{aligned}
\end{equation}
From $u \in \mathrm{VI}(C, \mathcal{A})$, one has $\langle \mathcal{A} u, y_{n}-u\rangle \geq 0$. Using the pseudomonotonicity of $\mathcal{A}$, we get
\begin{equation}\label{s}
\langle \mathcal{A} y_{n}, y_{n}-u\rangle \geq 0\,.
\end{equation}
Combining \eqref{w} and \eqref{s}, we can show that
\[
\|z_{n}-u\|^{2} \leq\|s_{n}-u\|^{2}-\Big(1-\phi^{2} \frac{\gamma_{n}^{2}}{\gamma_{n+1}^{2}}\Big)\|s_{n}-y_{n}\|^{2}\,.
\]
According to the definition of $ z_{n} $ and \eqref{q}, we obtain
\[
\|z_{n}-y_{n}\| \leq \phi \frac{\gamma_{n}}{\gamma_{n+1}}\|s_{n}-y_{n}\|\,.
\]
This completes the proof of the Lemma~\ref{lem33}. 
\end{proof}

\begin{theorem}\label{thm41}
Suppose that Assumptions \ref{con1}--\ref{con5} hold. Then the iterative sequence $\{x_{n}\}$ formulated by Algorithm~\ref{alg1} converges to  $u \in  \mathrm{VI}(C, \mathcal{A})$  in norm, where $u = P_{\mathrm{VI}(C,\mathcal{A})} \circ f(u)$.
\end{theorem}
\begin{proof}
\noindent \textbf{Claim 1.} The sequence $\{x_{n}\}$ is bounded. According to  Lemma~\ref{lem33}, we get that $\lim _{n \rightarrow \infty}\big(1-\phi^{2} \frac{\gamma_{n}^{2}}{\gamma_{n+1}^{2}}\big)=1-\phi^{2}>0$. Therefore, there is a constant $n_{0} \in \mathbb{N}$ that satisfies
$
1-\phi^{2} \frac{\gamma_{n}^{2}}{\gamma_{n+1}^{2}}>0, \forall n \geq n_{0}\,.
$
From Lemma~\ref{lem33}, one has
\begin{equation}\label{py}
\|z_{n}-u\| \leq\|s_{n}-u\|, \quad \forall n \geq n_{0}\,.
\end{equation}
By the definition of $s_{n}$, one sees that
\begin{equation}\label{pl}
\begin{aligned}
\|s_{n}-u\| &=\|x_{n}+\delta_{n}\left(x_{n}-x_{n-1}\right)-u\| \\
& \leq\|x_{n}-u\|+\delta_{n}\|x_{n}-x_{n-1}\| \\
&=\|x_{n}-u\|+\varphi_{n} \cdot \frac{\delta_{n}}{\varphi_{n}}\|x_{n}-x_{n-1}\|\,.
\end{aligned}
\end{equation}
From Remark~\ref{rem31}, one gets $\frac{\delta_{n}}{\varphi_{n}}\|x_{n}-x_{n-1}\| \rightarrow 0$. Thus, there is a constant $Q_{1}>0$ that satisfies
\begin{equation}\label{ppl}
\frac{\delta_{n}}{\varphi_{n}}\|x_{n}-x_{n-1}\| \leq Q_{1},\quad  \forall n \geq 1\,.
\end{equation}
Using \eqref{py}, \eqref{pl} and \eqref{ppl}, we obtain
\begin{equation}\label{pppl}
\|z_{n}-u\| \leq\|s_{n}-u\| \leq\|x_{n}-u\|+\varphi_{n} Q_{1},\quad  \forall n \geq n_{0}\,.
\end{equation}
Using the definition of $\{x_{n+1}\}$ and \eqref{pppl}, we have
\[
\begin{aligned}
\|x_{n+1}-u\| &=\|\varphi_{n} f\left(z_{n}\right)+\left(1-\varphi_{n}\right) z_{n}-u\| \\
& \leq \varphi_{n}\|f\left(z_{n}\right)-f(u)\|+\varphi_{n}\|f(u)-u\|+\left(1-\varphi_{n}\right)\|z_{n}-u\| \\
& \leq \varphi_{n} \rho\|z_{n}-u\|+\varphi_{n}\|f(u)-u\|+\left(1-\varphi_{n}\right)\|z_{n}-u\| \\
&=\left(1-(1-\rho) \varphi_{n}\right)\|z_{n}-u\|+\varphi_{n}\|f(u)-u\|\\
& \leq\left(1-(1-\rho) \varphi_{n}\right)\|x_{n}-u\|+\varphi_{n} Q_{1}+\varphi_{n}\|f(u)-u\|  \\
&=\left(1-(1-\rho) \varphi_{n}\right)\|x_{n}-u\|+(1-\rho) \varphi_{n} \frac{Q_{1}+\|f(u)-u\|}{1-\rho}  \\
& \leq \max \Big\{\|x_{n}-u\|, \frac{Q_{1}+\|f(u)-u\|}{1-\rho}\Big\}  \\
& \leq \cdots  \leq \max \Big\{\|x_{n_{0}}-u\|, \frac{Q_{1}+\|f(u)-u\|}{1-\rho}\Big\},\,\, \forall n \geq n_{0}\,.
\end{aligned}
\]
That is, $\{x_{n}\}$ is bounded. We have that $ \{s_{n}\} $,  $\{z_{n}\}$  and $\{f\left(z_{n}\right)\}$ are also bounded.

\noindent \textbf{Claim 2.}
\[
\Big(1-\phi^{2} \frac{\gamma_{n}^{2}}{\gamma_{n+1}^{2}}\Big)\|s_{n}-y_{n}\|^{2}\leq\|x_{n}-u\|^{2}-\|x_{n+1}-u\|^{2}+\varphi_{n} Q_{4}
\]
for some $Q_{4}>0$. Indeed, it follows from \eqref{pppl} that
\begin{equation}\label{lk}
\begin{aligned}
\|s_{n}-u\|^{2} & \leq(\|x_{n}-u\|+\varphi_{n} Q_{1})^{2} \\
&=\|x_{n}-u\|^{2}+\varphi_{n}(2 Q_{1}\|x_{n}-u\|+\varphi_{n} Q_{1}^{2}) \\
& \leq\|x_{n}-u\|^{2}+\varphi_{n} Q_{2}
\end{aligned}
\end{equation}
for some $Q_{2}>0$. Combining Lemma \ref{lem33} and \eqref{lk}, we see that
\begin{equation}\label{plm}
\begin{aligned}
\|x_{n+1}-u\|^{2} & \leq \varphi_{n}\|f\left(z_{n}\right)-u\|^{2}+\left(1-\varphi_{n}\right)\|z_{n}-u\|^{2} \\
& \leq \varphi_{n}(\|f\left(z_{n}\right)-f(u)\|+\|f(u)-u\|)^{2}+\left(1-\varphi_{n}\right)\|z_{n}-u\|^{2} \\
&\leq \varphi_{n}(\|z_{n}-u\|+\|f(u)-u\|)^{2}+\left(1-\varphi_{n}\right)\|z_{n}-u\|^{2} \\
&=\varphi_{n}\|z_{n}-u\|^{2}+\left(1-\varphi_{n}\right)\|z_{n}-u\|^{2}\\
&\quad+\varphi_{n}(\|f(u)-u\|^{2}+2\|z_{n}-u\| \cdot\|f(u)-u\|) \\
&\leq\|z_{n}-u\|^{2}+\varphi_{n} Q_{3}\\
&\leq \|s_{n}-u\|^{2}-\Big(1-\phi^{2} \frac{\gamma_{n}^{2}}{\gamma_{n+1}^{2}}\Big)\|s_{n}-y_{n}\|^{2}+\varphi_{n} Q_{3}\\
&\leq \|x_{n}-u\|^{2}-\Big(1-\phi^{2} \frac{\gamma_{n}^{2}}{\gamma_{n+1}^{2}}\Big)\|s_{n}-y_{n}\|^{2}+\varphi_{n} Q_{4}
\end{aligned}
\end{equation}
where $Q_{4}:=Q_{2}+Q_{3}$. Therefore, we obtain
\[
\Big(1-\phi^{2} \frac{\gamma_{n}^{2}}{\gamma_{n+1}^{2}}\Big)\|s_{n}-y_{n}\|^{2}\leq\|x_{n}-u\|^{2}-\|x_{n+1}-u\|^{2}+\varphi_{n} Q_{4}\,.
\]

\noindent \textbf{Claim 3.}
\[
\begin{aligned}
\|x_{n+1}-u\|^{2} \leq& \left(1-(1-\rho) \varphi_{n}\right)\|x_{n}-u\|^{2}+(1-\rho) \varphi_{n}\cdot\Big[\frac{3 Q}{1-\rho} \cdot \frac{\delta_{n}}{\varphi_{n}}\|x_{n}-x_{n-1}\|\Big. \\
&+\Big.\frac{2}{1-\rho}\langle f(u)-u, x_{n+1}-u\rangle\Big],\quad \forall n \geq n_{0}\,.
\end{aligned}
\]
for some $Q>0$. Using the definition of $ s_{n} $, we can show that
\begin{equation}\label{pm}
\begin{aligned}
\|s_{n}-u\|^{2}
&=\|x_{n}+\delta_{n}\left(x_{n}-x_{n-1}\right)-u\|^{2} \\
&\leq\|x_{n}-u\|^{2}+2 \delta_{n}\|x_{n}-u\|\|x_{n}-x_{n-1}\|+\delta_{n}^{2}\|x_{n}-x_{n-1}\|^{2}\\
&\leq\|x_{n}-u\|^{2}+3Q\delta_{n}\|x_{n}-x_{n-1}\|\,,
\end{aligned}
\end{equation}
where $Q:=\sup _{n \in \mathbb{N}}\{\|x_{n}-u\|, \delta\|x_{n}-x_{n-1}\|\}>0$.
Using \eqref{py} and \eqref{pm}, we get
\begin{equation}
\begin{aligned}
&\quad \|x_{n+1}-u\|^{2} =\|\varphi_{n} f\left(z_{n}\right)+\left(1-\varphi_{n}\right) z_{n}-u\|^{2} \\
&=\|\varphi_{n}(f\left(z_{n}\right)-f(u))+\left(1-\varphi_{n}\right)(z_{n}-u)+\varphi_{n}(f(u)-u)\|^{2} \\
& \leq\|\varphi_{n}(f\left(z_{n}\right)-f(u))+\left(1-\varphi_{n}\right)(z_{n}-u)\|^{2}+2 \varphi_{n}\langle f(u)-u, x_{n+1}-u\rangle \\
& \leq \varphi_{n}\|f\left(z_{n}\right)-f(u)\|^{2}+\left(1-\varphi_{n}\right)\|z_{n}-u\|^{2}+2 \varphi_{n}\langle f(u)-u, x_{n+1}-u\rangle \\
& \leq \varphi_{n} \rho^{2}\|z_{n}-u\|^{2}+\left(1-\varphi_{n}\right)\|z_{n}-u\|^{2}+2 \varphi_{n}\langle f(u)-u, x_{n+1}-u\rangle \\
&\leq\left(1-(1-\rho) \varphi_{n}\right)\|z_{n}-u\|^{2}+2 \varphi_{n}\langle f(u)-u, x_{n+1}-u\rangle \\
&\leq \left(1-(1-\rho) \varphi_{n}\right)\|x_{n}-u\|^{2}+(1-\rho) \varphi_{n}\cdot\Big[\frac{3 Q}{1-\rho} \cdot \frac{\delta_{n}}{\varphi_{n}}\|x_{n}-x_{n-1}\|\Big. \\
&\quad+ \Big.\frac{2}{1-\rho}\langle f(u)-u, x_{n+1}-u\rangle\Big],\quad \forall n \geq n_{0}\,.
\end{aligned}
\end{equation}

\noindent \textbf{Claim 4.} $\{\|x_{n}-u\|^{2}\}$ converges to zero. From Lemma~\ref{lem22} and Remark~\ref{rem31},  it remains to show that $\lim \sup _{k \rightarrow \infty}\langle f(u)-u, x_{n_{k}+1}-u\rangle \leq 0$ for any subsequence $\{\|x_{n_{k}}-u\|\}$ of
$\{\|x_{n}-u\|\}$ satisfies $ \liminf _{k \rightarrow \infty}\big(\|x_{n_{k}+1}-u\|-\|x_{n_{k}}-u\|\big) \geq 0 $.

For this purpose, we assume that $\{\|x_{n_{k}}-u\|\}$ is a subsequence of $\{\|x_{n}-u\|\}$ such that
\[
\liminf _{k \rightarrow \infty}\left(\|x_{n_{k}+1}-u\|-\|x_{n_{k}}-u\|\right) \geq 0\,.
\]
Then,
\[\begin{aligned}
&\quad\lim _{k \rightarrow \infty} \inf \big(\|x_{n_{k}+1}-u\|^{2}-\|x_{n_{k}}-u\|^{2}\big) \\
&=\liminf _{k \rightarrow \infty}\big[(\|x_{n_{k}+1}-u\|-\|x_{n_{k}}-u\|)(\|x_{n_{k}+1}-u\|+\|x_{n_{k}}-u\|)\big] \geq 0\,.
\end{aligned}
\]
It follows from Claim 2 and Assumption \ref{con5} that
\[
\begin{aligned}
&\quad \limsup _{k \rightarrow \infty}\big(1-\phi^{2} \frac{\gamma_{n_{k}}^{2}}{\gamma_{n_{k}+1}^{2}}\big)\|s_{n_{k}}-y_{n_{k}}\|^{2} \\
& \leq \limsup _{k \rightarrow \infty}\big[\|x_{n_{k}}-u\|^{2}-\|x_{n_{k}+1}-u\|^{2}\big]+\limsup _{k \rightarrow \infty} \varphi_{n_{k}} Q_{4} \\
&=-\liminf _{k \rightarrow \infty}\big[\|x_{n_{k}+1}-u\|^{2}-\|x_{n_{k}}-u\|^{2}\big] \\
& \leq 0\,,
\end{aligned}
\]
which yields that $ \lim _{k \rightarrow \infty}\|s_{n_{k}}-y_{n_{k}}\|=0$. From Lemma~\ref{lem33}, we obtain $\lim _{k \rightarrow \infty}\|z_{n_{k}}-y_{n_{k}}\|=0 $.
Hence,   $\lim _{k \rightarrow \infty}\|z_{n_{k}}-s_{n_{k}}\|=0$.

Moreover, using Remark \ref{rem31} and Assumption \ref{con5}, we have
\[
\|x_{n_{k}}-s_{n_{k}}\|=\delta_{n_{k}}\|x_{n_{k}}-x_{n_{k}-1}\|=\varphi_{n_{k}} \cdot \frac{\delta_{n_{k}}}{\varphi_{n_{k}}}\|x_{n_{k}}-x_{n_{k}-1}\| \rightarrow 0\,,
\]
and
\[
\|x_{n_{k}+1}-z_{n_{k}}\|=\varphi_{n_{k}}\|z_{n_{k}}-f\left(z_{n_{k}}\right)\| \rightarrow 0\,.
\]
Therefore, we conclude that
\begin{equation}\label{pj}
\|x_{n_{k}+1}-x_{n_{k}}\| \leq\|x_{n_{k}+1}-z_{n_{k}}\|+\|z_{n_{k}}-s_{n_{k}}\|+\|s_{n_{k}}-x_{n_{k}}\| \rightarrow 0\,.
\end{equation}
Since  $\{x_{n_{k}}\}$ is bounded, one asserts that there is a subsequence $\{x_{n_{k_{j}}}\}$ of $\{x_{n_{k}}\}$ that satisfies $ x_{n_{k_{j}}}\rightharpoonup q$. Furthermore,
\begin{equation}\label{pk}
\limsup _{k \rightarrow \infty}\langle f(u)-u, x_{n_{k}}-u\rangle=\lim _{j \rightarrow \infty}\langle f(u)-u, x_{n_{k_{j}}}-u\rangle=\langle f(u)-u, q-u\rangle\,.
\end{equation}
We get $s_{n_{k}} \rightharpoonup q$ since $ \|x_{n_{k}}-s_{n_{k}}\|\rightarrow 0 $. This together with $ \lim _{k \rightarrow \infty}\|s_{n_{k}}-y_{n_{k}}\|=0$ and Lemma~\ref{lem32} obtains $q \in \mathrm{VI}(C, \mathcal{A}) $. By the definition of $ u = P_{\mathrm{VI}(C,\mathcal{A})} \circ f(u) $ and \eqref{pk}, we infer that
\begin{equation}\label{pn}
\limsup _{k \rightarrow \infty}\langle f(u)-u, x_{n_{k}}-u\rangle=\langle f(u)-u, q-u\rangle \leq 0\,.
\end{equation}
Combining \eqref{pj} and \eqref{pn},  we see that
\begin{equation}\label{pb}
\begin{aligned}
\limsup _{k \rightarrow \infty}\langle f(u)-u, x_{n_{k}+1}-u\rangle & \leq \limsup _{k \rightarrow \infty}\langle f(u)-u, x_{n_{k}}-u\rangle \leq 0\,.
\end{aligned}
\end{equation}
Thus, from Remark~\ref{rem31}, \eqref{pb}, Claim 3 and Lemma~\ref{lem22}, we conclude that $ x_{n}\rightarrow u $.  The proof of the Theorem~\ref{thm41} is now complete. 
\end{proof}

If inertial parameter $ \delta_{n}=0 $ in Algorithm~\ref{alg1}, we have the following result.
\begin{corollary}
Assume that mapping $\mathcal{A}: H \rightarrow H$ is $ L $-Lipschitz continuous pseudomonotone on $H$ and sequentially weakly continuous on $C$. Let mapping $f: H \rightarrow H$ be $\rho$-contractive with $\rho \in[0,1)$. Given $\gamma_{0}>0$, $ \{\varphi_{n}\}\subset (0,1) $ satisfies
$\lim _{n \rightarrow \infty} \varphi_{n}=0$ and $\sum_{n=1}^{\infty} \varphi_{n}=\infty$. Let $x_{0}$  be the initial point and $ \{x_{n}\} $ be the sequence  generated by
\begin{equation}\label{coro1}
\left\{\begin{aligned}
&y_{n}=P_{C}\left(x_{n}-\gamma_{n} \mathcal{A} x_{n}\right) \,,\\
&z_{n}=y_{n}-\gamma_{n}\left(\mathcal{A} y_{n}-\mathcal{A} x_{n}\right)\,,\\
&x_{n+1}=\varphi_{n} f\left(z_{n}\right) + \left(1-\varphi_{n}\right) z_{n}\,,
\end{aligned}\right.
\end{equation}
where step size $ \{\gamma_{n}\} $ is updated through \eqref{lambda}. Then the iterative sequence $\{x_{n}\}$ formulated by Algorithm~\eqref{coro1} converges to  $u \in  \mathrm{VI}(C, \mathcal{A})$  in norm, where $u = P_{\mathrm{VI}(C,\mathcal{A})} \circ f(u)$.
\end{corollary}

\begin{remark}
It should be pointed out that Algorithm~\eqref{coro1} improves and summarizes \cite[Algorithm 3]{THna} and \cite[Algorithm 1]{YLNA}. Moreover, our algorithm is to solve pseudomonotone \eqref{VIP}, while \cite{THna} and \cite{YLNA} are to solve monotone \eqref{VIP}. We know that the classes of pseudomonotone mappings cover the classes of monotone mappings. Therefore, our algorithm is more applicable.
\end{remark}
\section{Numerical examples}\label{sec4}
In this section, we give some computational tests and applications to show the numerical behavior of our algorithm, and also to compare it with some strong convergent algorithms (Algorithms~\eqref{MaTEGM} and \eqref{ViSEGM}).  It should be emphasized that all algorithms can work without the prior information of the Lipschitz constant of the mapping. We use the FOM Solver~\cite{FOM} to effectively calculate the projections onto $ C $ and $ T_{n} $. All the programs are implemented in MATLAB 2018a on a personal computer. The parameters are chosen as follows:
\begin{itemize}
\item $ \phi=0.8 $, $ \gamma_{1}=1 $, $ \delta=0.3 $, $ \epsilon_{n}=1/(n+1)^2 $, $ \varphi_{n}=1/(n+1) $, $ f(x)=0.9x $ for the proposed Algorithm~\ref{alg1} and the Algorithm~\eqref{ViSEGM};
\item  $ \alpha=\ell=0.5 $, $ \phi=0.4 $, $ \varphi_{n}=1/(n+1) $, $ \tau_{n}=0.5(1-\varphi_{n}) $ for the Algorithm~\eqref{MaTEGM}.
\end{itemize}

In our numerical examples, when the number of iterations is the same, we use the runtime in seconds to measure the computational performance of all algorithms. In the situation, if the solution $x^{*}$ of our problem is known, we take $E(x)=\left\|x-x^{*}\right\|$ to represent the behavior of all algorithms. Otherwise,  according to the feature of solutions to~\eqref{VIP}, we use the sequences $ D_{n}=\|x_{n}-x_{n-1}\| $ and $ E_{n}=\|s_{n}-P_{C}(s_{n}-\gamma_{n}s_{n})\| $ to study the performance of all algorithms.  Note that, if $\left\|E_{n}\right\|\rightarrow 0$, then $x_{n}$ can be regards as an approximate solution of~\eqref{VIP}.

\begin{example}\label{ex1}
Let ${\mathcal{A}}: R^{m} \rightarrow R^{m}\, (m=5,10,15,20)$ be  an operator given by
\[
{{\mathcal{A}}}(x)=\frac{1}{\|x\|^{2}+1} \operatorname{argmin}_{y \in R^{m}} \Big\{\frac{\|y\|^{4}}{4}+\frac{1}{2}\|x-y\|^{2}\Big\}\,.
\]
We emphasize that the operator ${\mathcal{A}}$ is not monotone. However, the operator $\mathcal{A}$ is Lipschitz continuous pseudomonotone (see~\cite{HCXK}). In this example, we choose the feasible set is a box constraint $ C=[-5,5]^{m} $. Take initial values $ x_{0} = x_{1} $ are randomly generated by \emph{rand(m,1)} in MATLAB. The maximum iteration $ 50 $ as a common stopping criterion. For the four different dimensions of the operator $ \mathcal{A} $, the numerical results are presented in Figs.~\ref{ex1_data5}--\ref{ex1_data20}.
\begin{figure}[H]
\centering
\subfigure{
\includegraphics[width=0.45\textwidth]{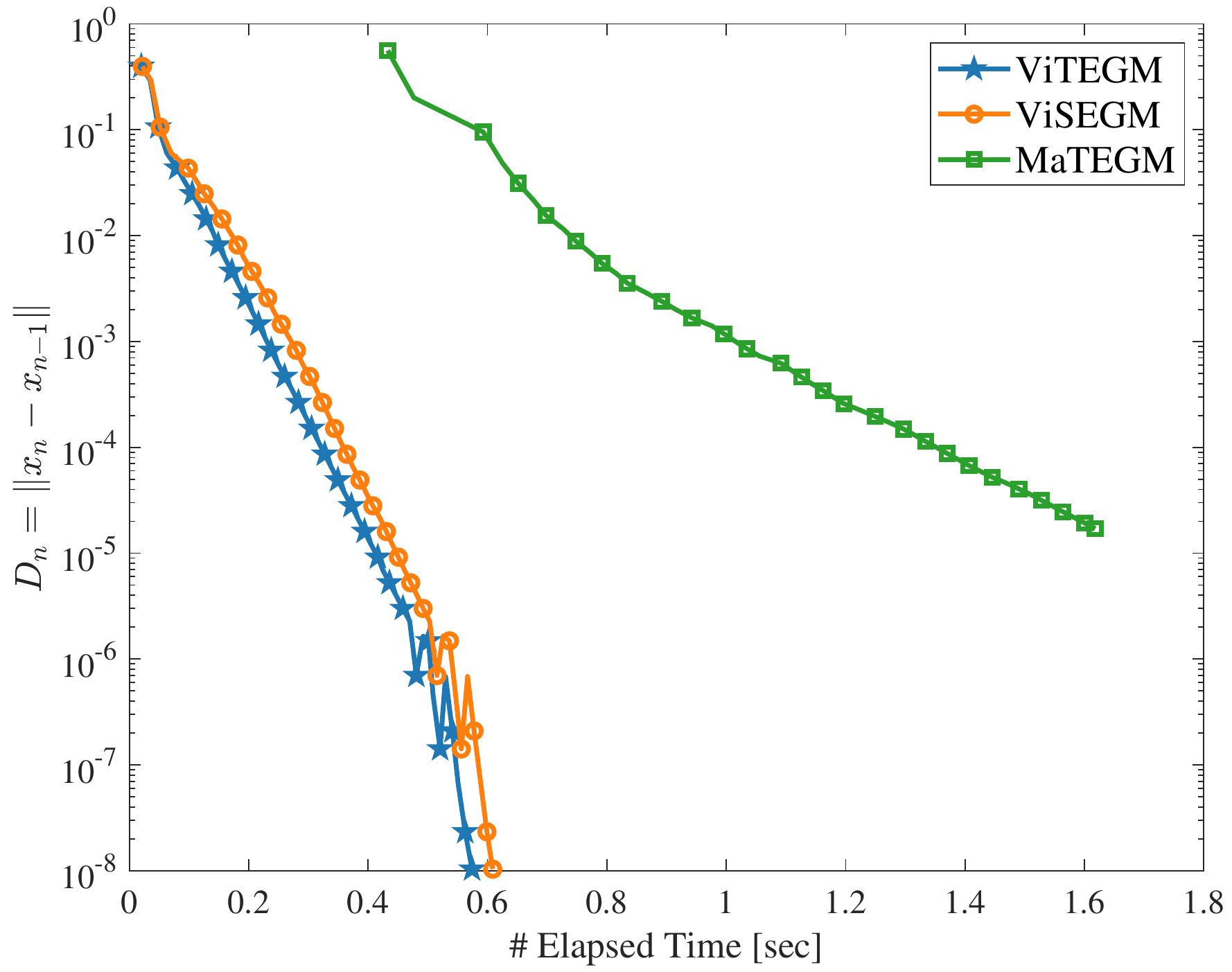}}
\subfigure{
\includegraphics[width=0.45\textwidth]{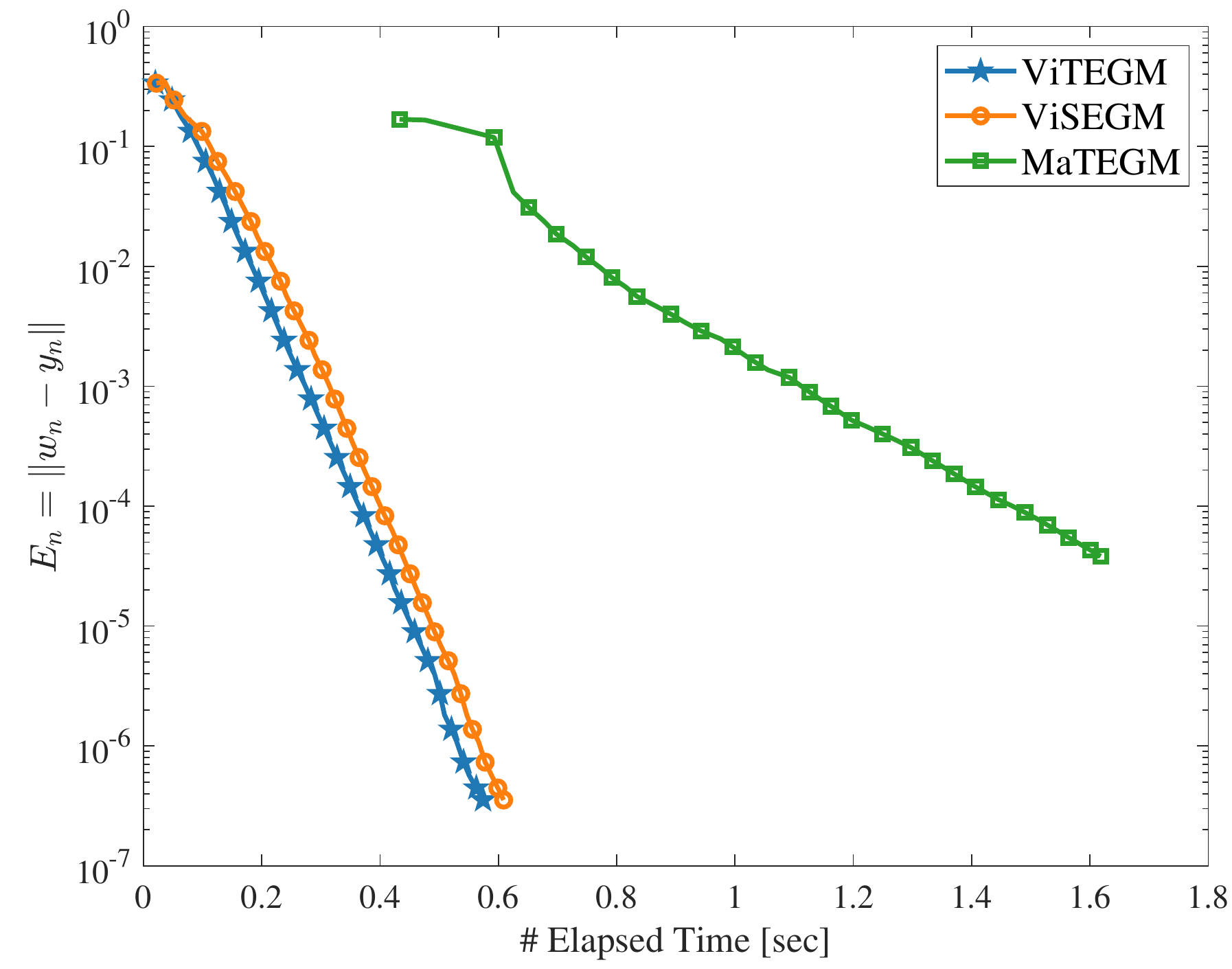}}
\caption{Numerical results for Example \ref{ex1} ($ m=5 $)}
\label{ex1_data5}
\end{figure}
\begin{figure}[htbp]
\centering
\subfigure{
\includegraphics[width=0.45\textwidth]{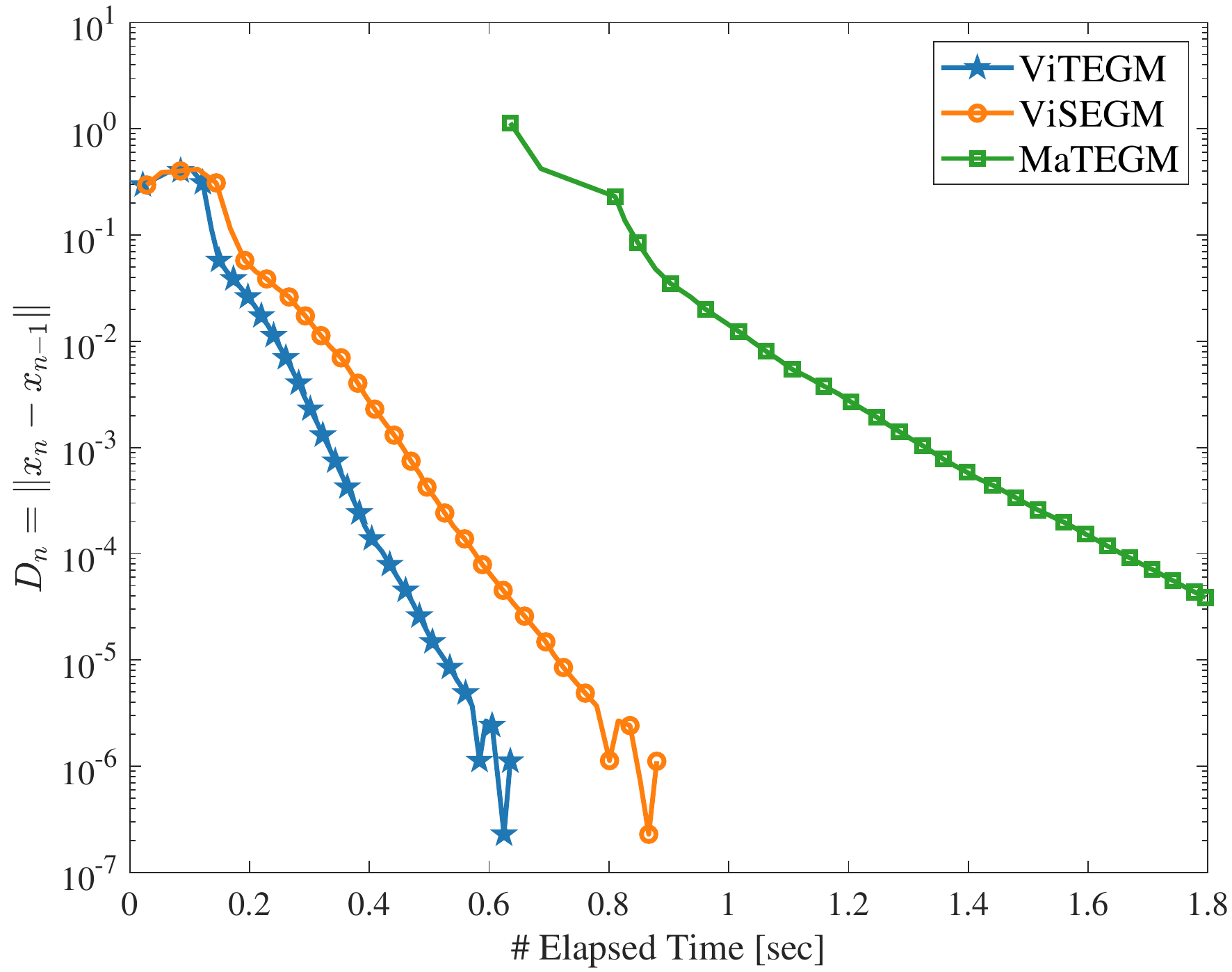}}
\subfigure{
\includegraphics[width=0.45\textwidth]{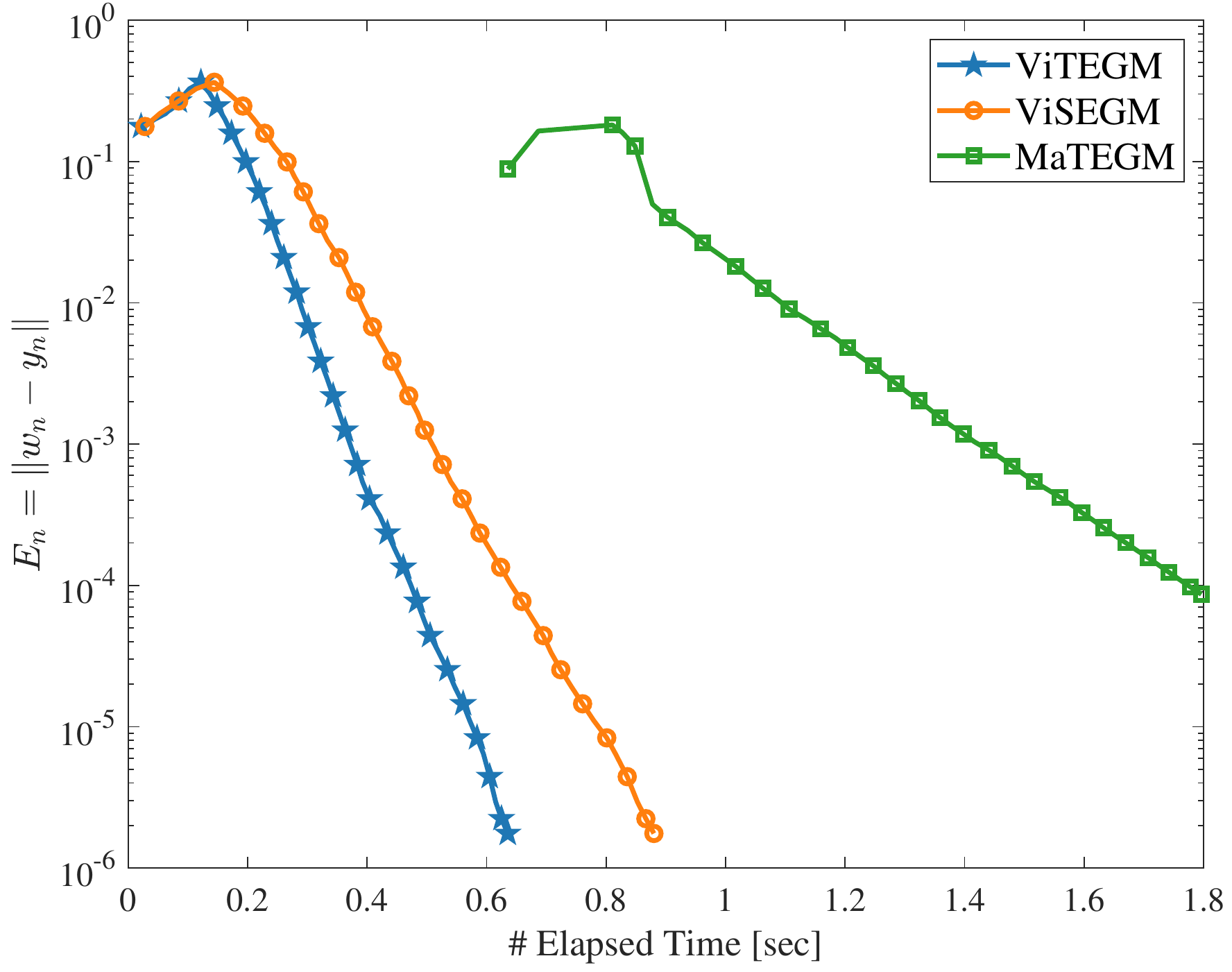}}
\caption{Numerical results for Example \ref{ex1} ($ m=10 $)}
\label{ex1_data10}
\end{figure}
\begin{figure}[htbp]
\centering
\subfigure{
\includegraphics[width=0.45\textwidth]{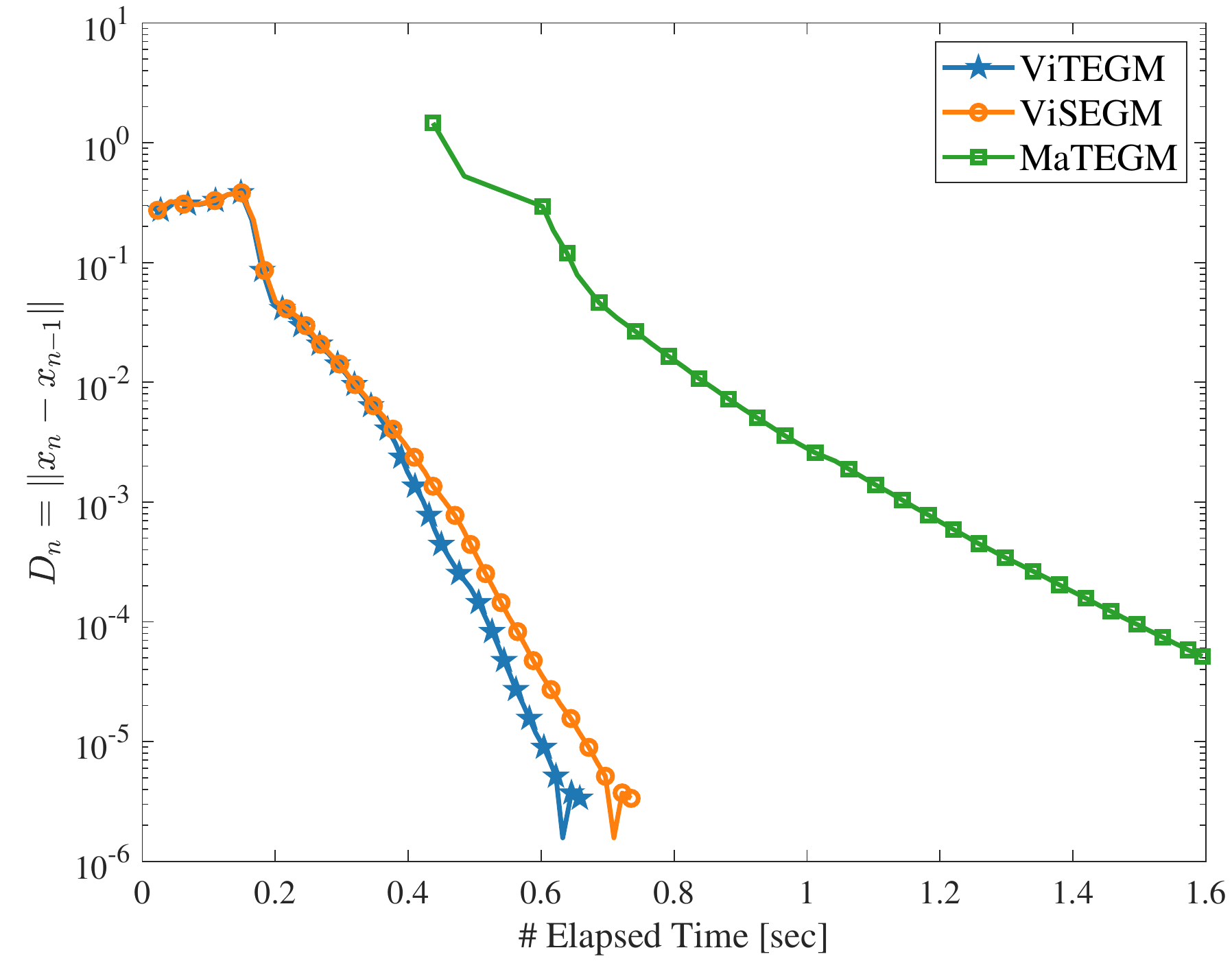}}
\subfigure{
\includegraphics[width=0.45\textwidth]{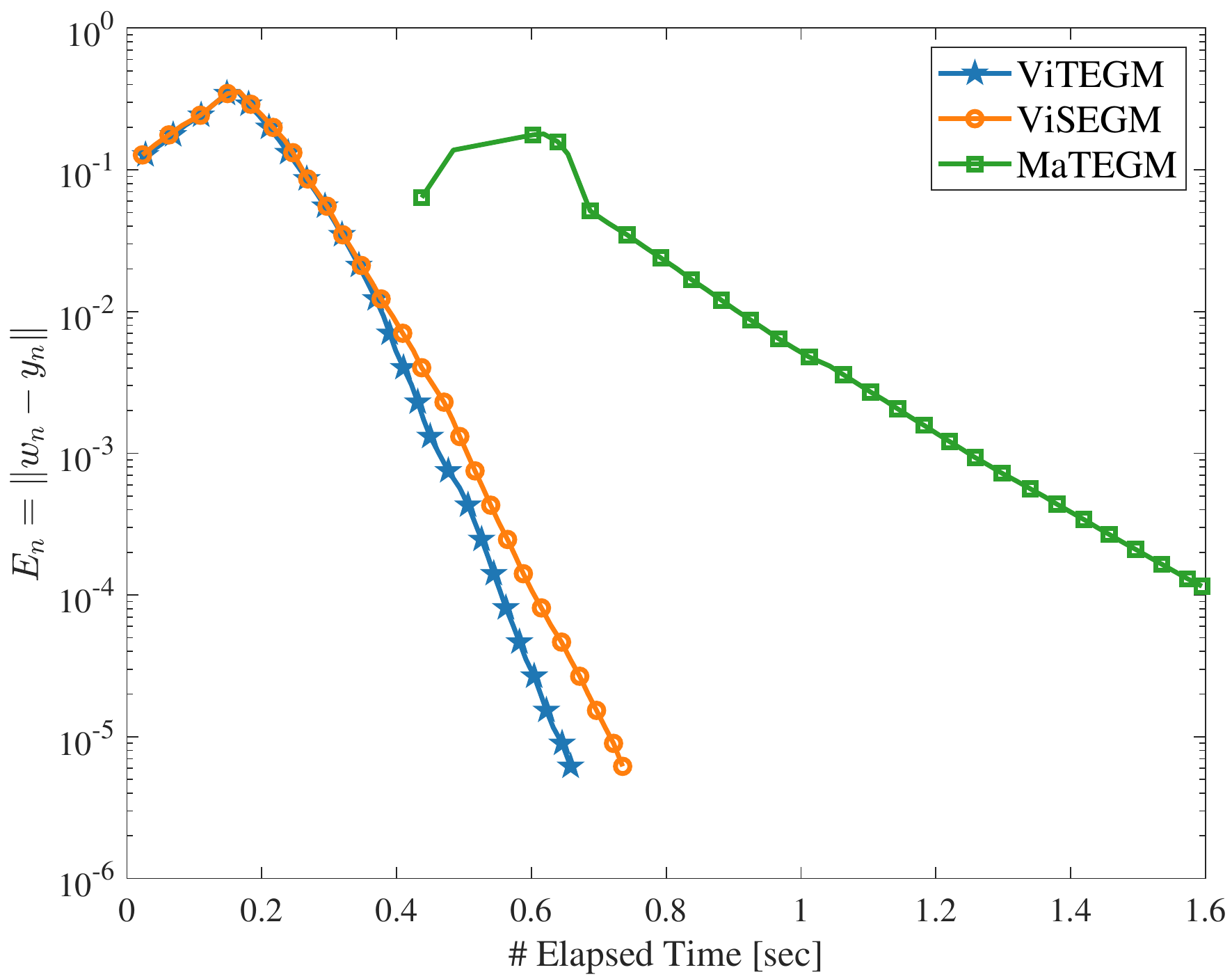}}
\caption{Numerical results for Example \ref{ex1} ($ m=15 $)}
\label{ex1_data15}
\end{figure}
\begin{figure}[htbp]
\centering
\subfigure{
\includegraphics[width=0.45\textwidth]{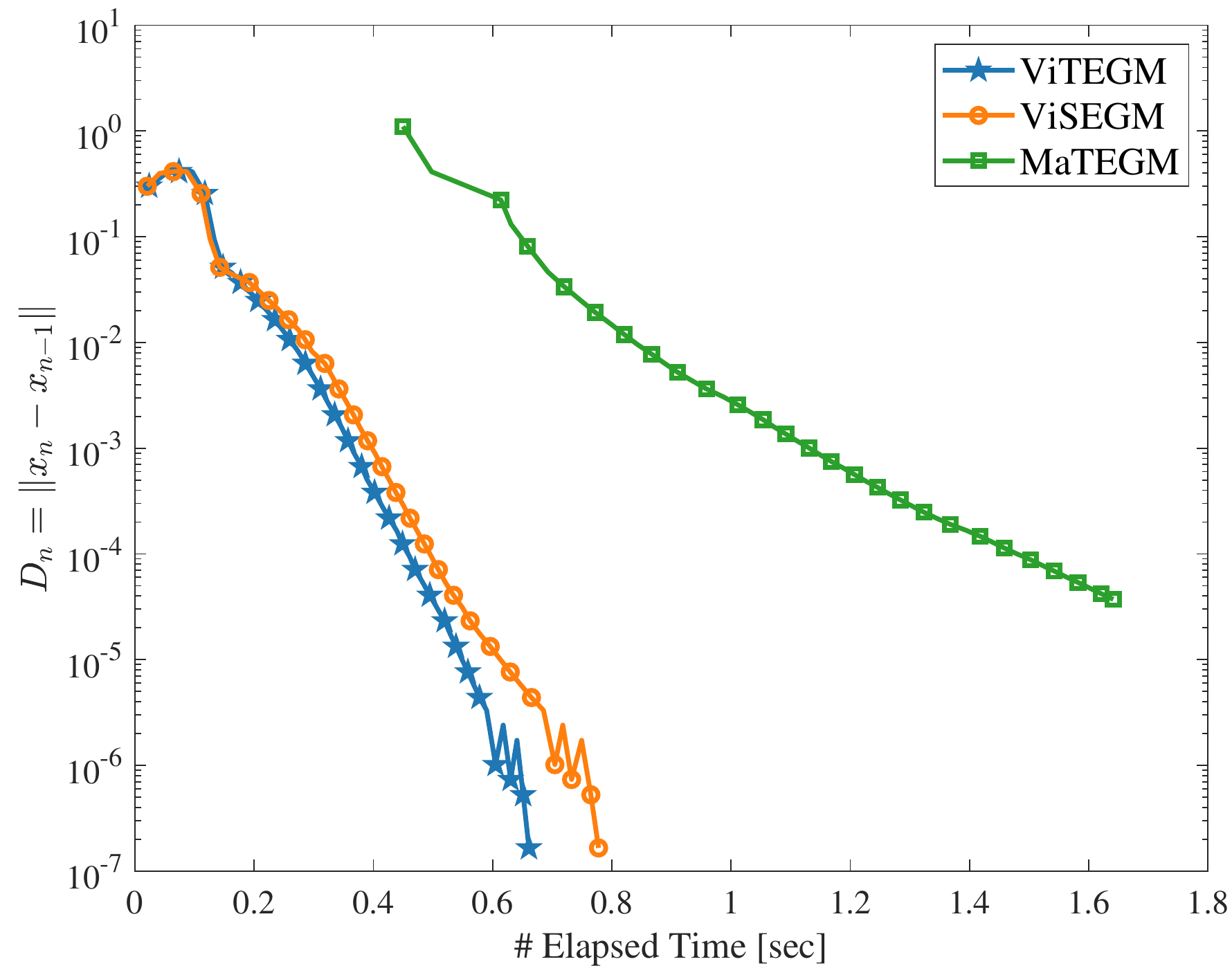}}
\subfigure{
\includegraphics[width=0.45\textwidth]{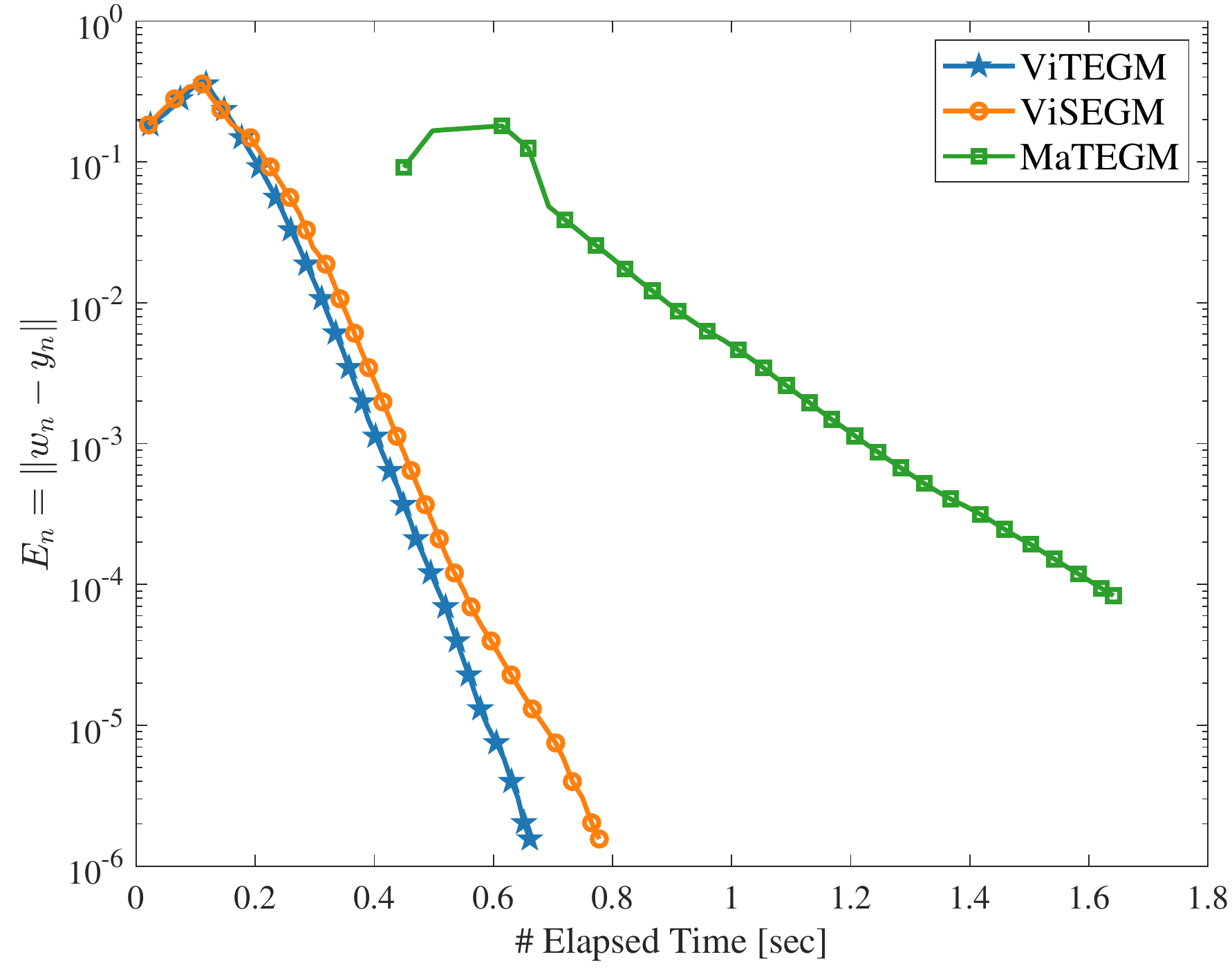}}
\caption{Numerical results for Example \ref{ex1} ($ m=20 $)}
\label{ex1_data20}
\end{figure}
\end{example}
\begin{example}\label{ex2}
In the second example, we consider the form of linear operator $ \mathcal{A}: R^{m}\rightarrow R^{m} $ ($ m=5,10,15,20 $) as follows: $\mathcal{A}(x)=Gx+g$, where $g\in R^{m}$ and $G=BB^{\mathsf{T}}+M+E$, matrix $B\in R^{m\times m}$, matrix $M\in R^{m\times m}$ is skew-symmetric, and matrix $E\in R^{m\times m}$ is diagonal matrix whose diagonal terms are non-negative (hence $ G $ is positive symmetric definite). We choose the feasible set as $C=\left\{x \in {R}^{m}:-2 \leq x_{i} \leq 5, \, i=1, \ldots, m\right\}$.  We get that mapping $\mathcal{A}$ is strongly pseudomonotone and Lipschitz continuous.  In this numerical example, both $B, M$ entries are randomly created in $[-2,2]$, $E$ is generated randomly in $[0,2]$ and $ g = \mathbf{0} $. It can be easily seen that the solution to the problem is $ x^{*}=\{\mathbf{0}\} $. The maximum iteration $ 1000 $ as a common stopping criterion and the initial values $ x_{0} = x_{1} $ are randomly generated by \emph{rand(m,1)} in MATLAB. The numerical results with elapsed time are described in Fig.~\ref{ex2_res}.
\begin{figure}[htbp]
\centering
\subfigure[$ m=5 $]{
\includegraphics[width=0.45\textwidth]{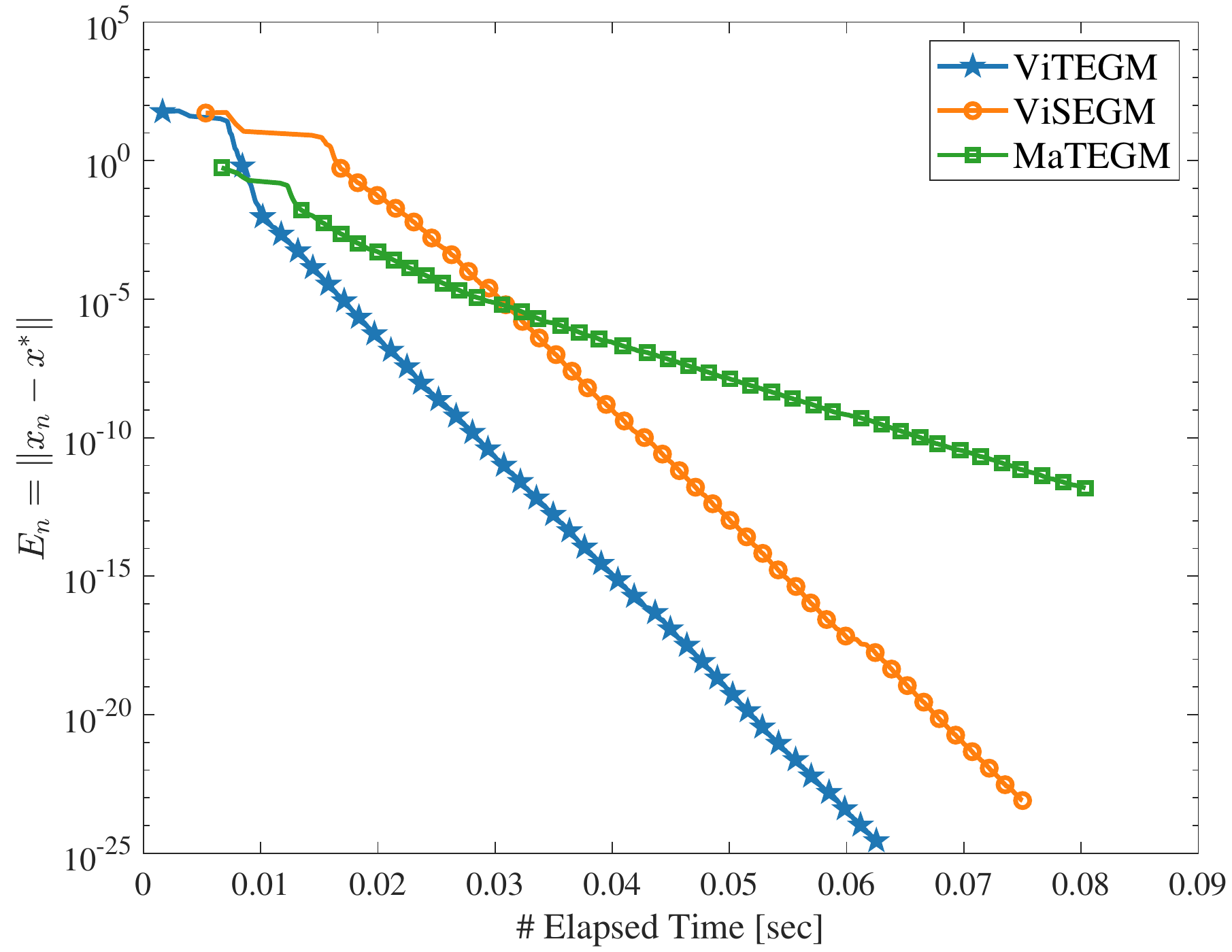}}
\subfigure[$ m=10 $]{
\includegraphics[width=0.45\textwidth]{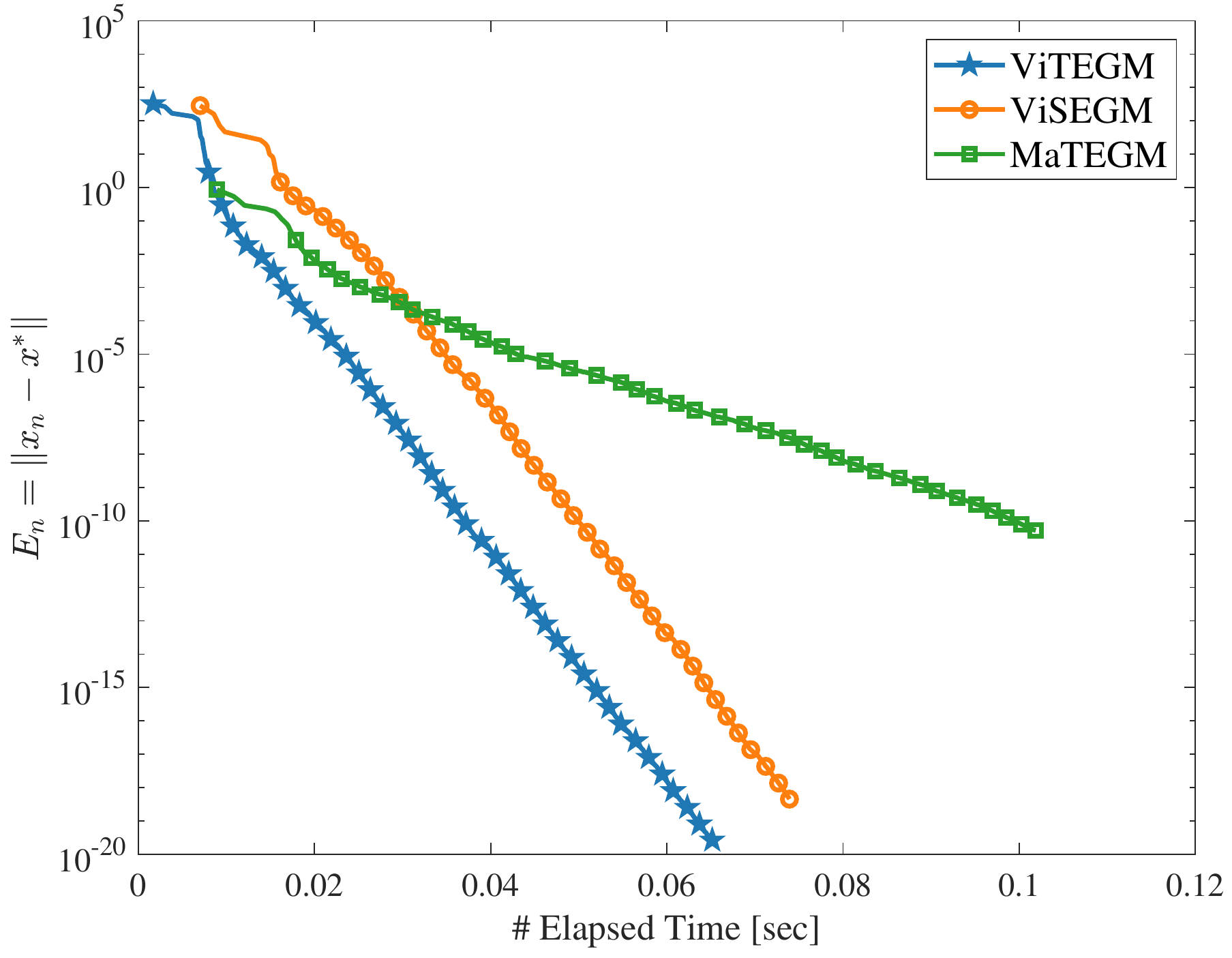}}
\subfigure[$ m=15 $]{
\includegraphics[width=0.45\textwidth]{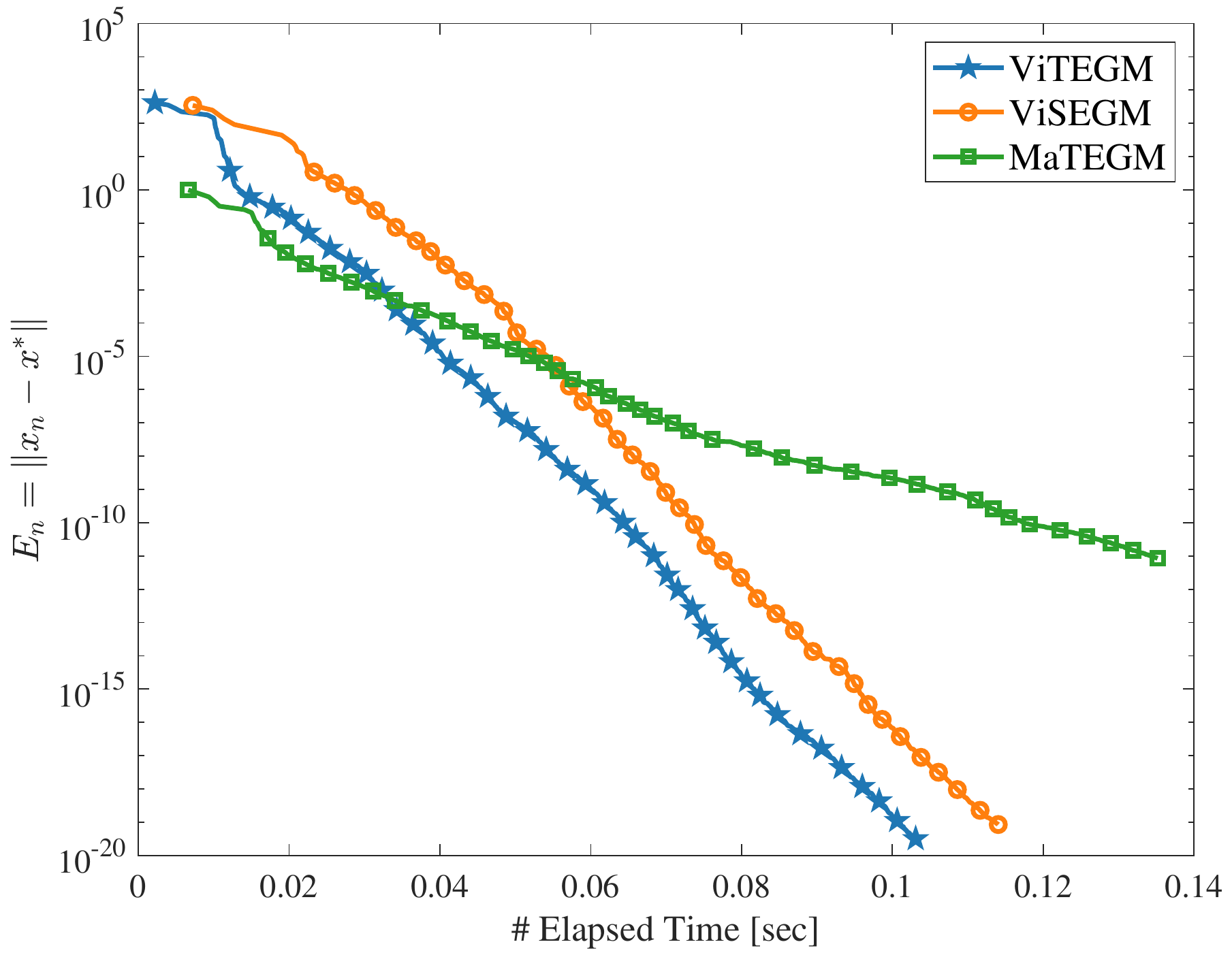}}
\subfigure[$ m=20 $]{
\includegraphics[width=0.45\textwidth]{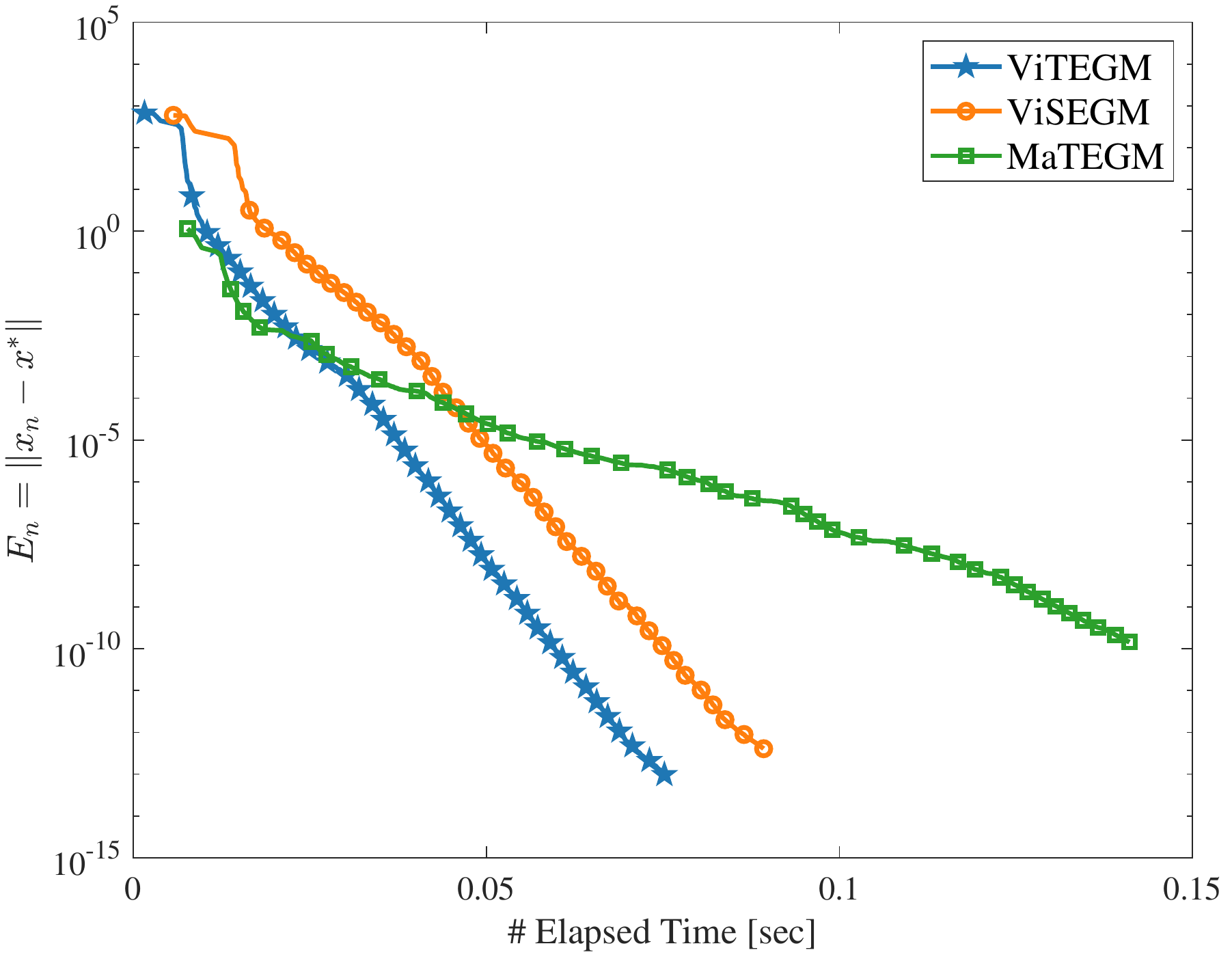}}
\caption{Numerical results for Example \ref{ex2}}
\label{ex2_res}
\end{figure}
\end{example}

\begin{example}\label{ex3}
Finally, we focus on a case in Hilbert space $H=L^{2}[0,1]$ with inner product
\[
\langle x, y\rangle=\int_{0}^{1} x(t) y(t) \mathrm{d} t\,,
\] and norm
\[
\|x\|=(\int_{0}^{1} x(t)^{2} \mathrm{d} t)^{1 / 2}\,.
\]
Let $b$ and $B$ be two positive numbers such that $B/(m+1)<b/m<b<B$ for some $m>1$. We select the feasible set as $C=\{x \in H:\|x\| \leq b\}$. The operator $\mathcal{A}: H \rightarrow H$ is of the form
\[
\mathcal{A}(x)=(B-\|x\|) x, \quad \forall x \in H\,.
\]
It should be pointed out that operator $\mathcal{A}$ is not monotone. Indeed, take a particular pair $(x^{\ddagger}, mx^{\ddagger})$, we pick $x^{\ddagger} \in C$ to satisfy $B /(m+1)<\|x^{\ddagger}\|<b / m$, one can sees that  $m\|x^{\ddagger}\| \in C $. By a simple operation, we get
\[
\langle \mathcal{A}(x^{\ddagger})-\mathcal{A}(y^{\ddagger}), x^{\ddagger}-y^{\ddagger}\rangle=(1-m)^{2}\|x^{\ddagger}\|^{2}(B-(1+m)\|x^{\ddagger}\|)<0\,.
\]
Hence, the operator $\mathcal{A}$ is not monotone on $C$. Next, we show that $\mathcal{A}$ is pseudomonotone. Indeed, one assumes that $\langle \mathcal{A}(x), y-x\rangle \geq 0, \forall x, y \in C$, that is, $\langle(B-\|x\|) x, y-x\rangle \geq 0$. From $\|x\|<B $, we get that $\langle x, y-x\rangle \geq 0$. Therefore, we can show that
\[
\begin{aligned}
\langle \mathcal{A}(y), y-x\rangle &=\langle(B-\|y\|) y, y-x\rangle \\
& \geq(B-\|y\|)(\langle y, y-x\rangle-\langle x, y-x\rangle) \\
&=(B-\|y\|)\|y-x\|^{2} \geq 0\,.
\end{aligned}
\]
For the experiment, we take $B=1.5$, $b=1$, $m=1.1$. We know that the solution to the problem is $x^{*}(t)=0 $. The maximum iteration $ 50 $ as the stopping criterion.  Fig.~\ref{ex3_res} shows the behaviors of function $E_{n}=\left\|x_{n}(t)-x^{*}(t)\right\|$ formulated by all algorithms with four initial points $ x_{0}(t)=x_{1}(t) $ (Case I: $x_{1}(t)=t^2$, Case II: $x_{1}(t)=\cos(t)$, Case III: $x_{1}(t)=\sin(2t)$ and  Case IV: $x_{1}(t)=2^{t}$).
\begin{figure}[htbp]
\centering
\subfigure[Starting points $x_{1}(t)=t^2$]{
\includegraphics[width=0.45\textwidth]{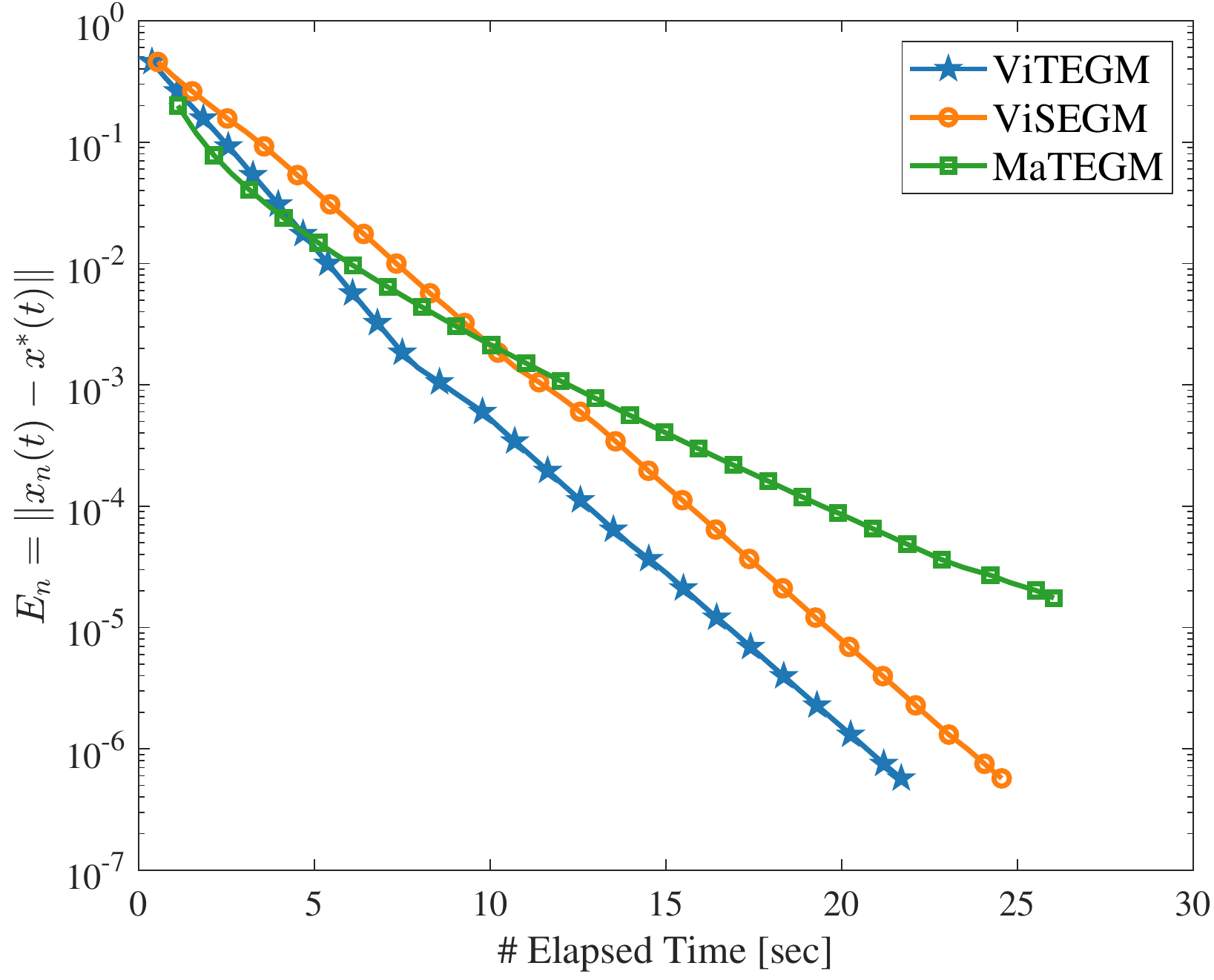}}
\subfigure[Starting points $x_{1}(t)=\cos(t)$]{
\includegraphics[width=0.45\textwidth]{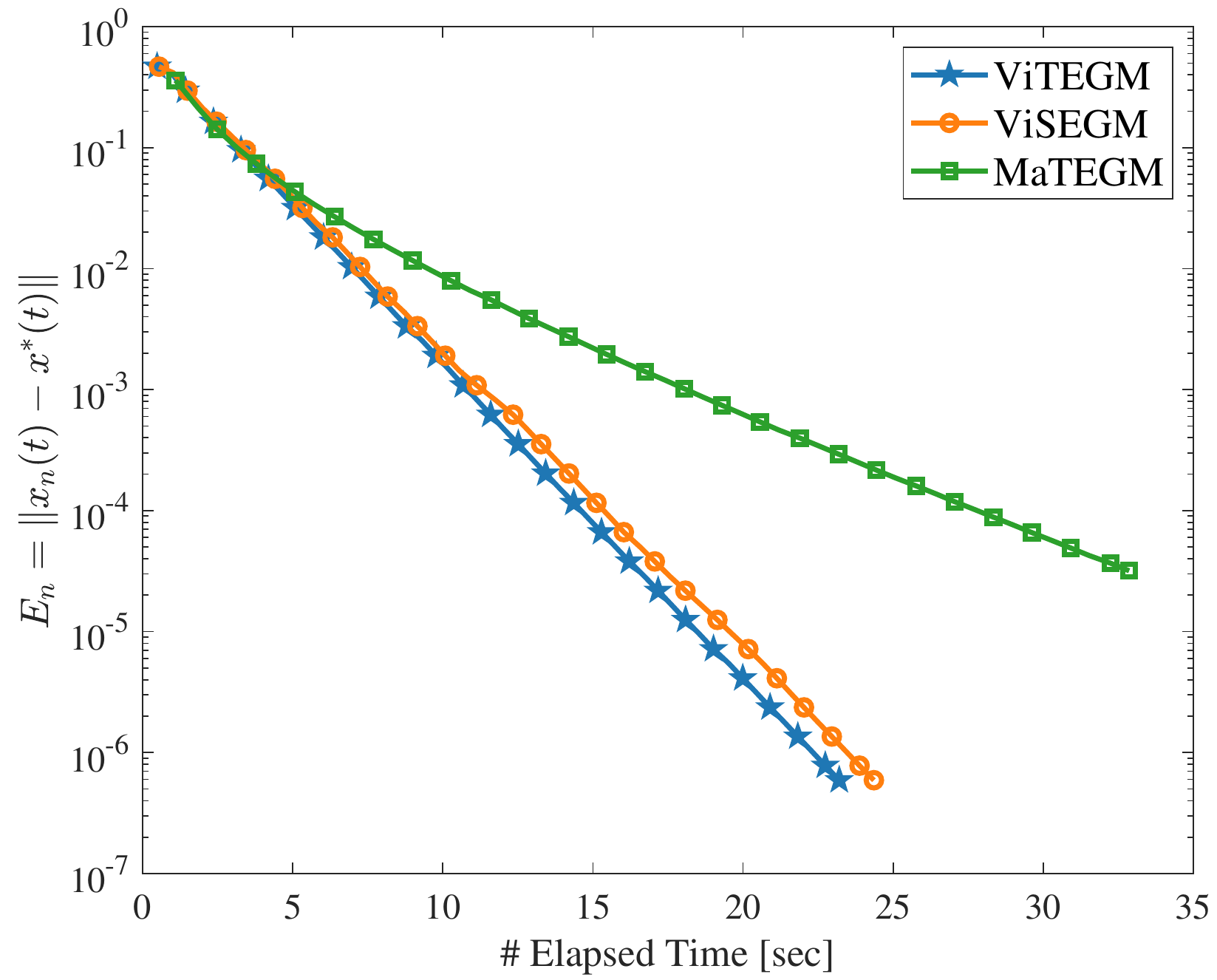}}
\subfigure[Starting points $x_{1}(t)=\sin(2t)$]{
\includegraphics[width=0.45\textwidth]{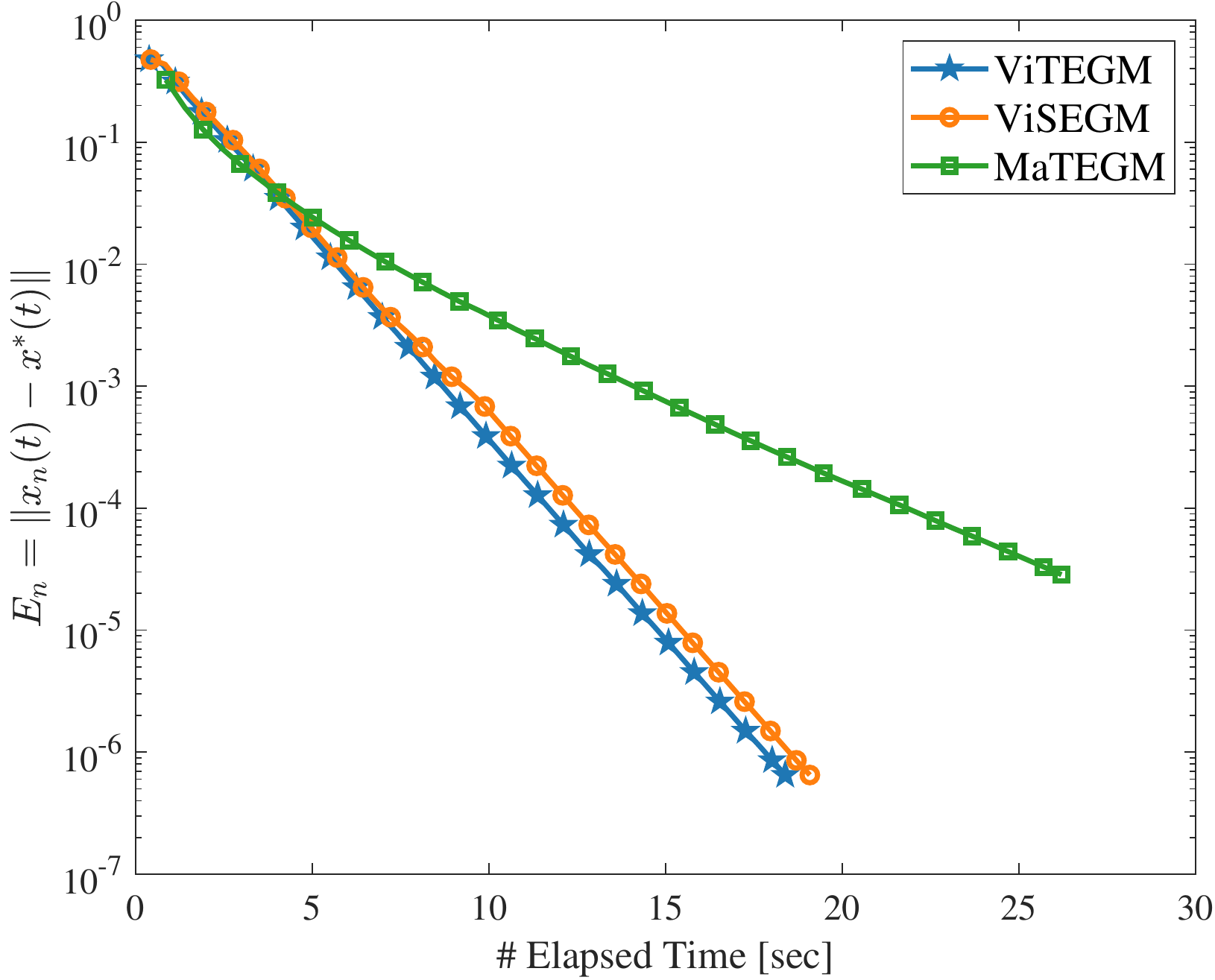}}
\subfigure[Starting points $x_{1}(t)=2^t$]{
\includegraphics[width=0.45\textwidth]{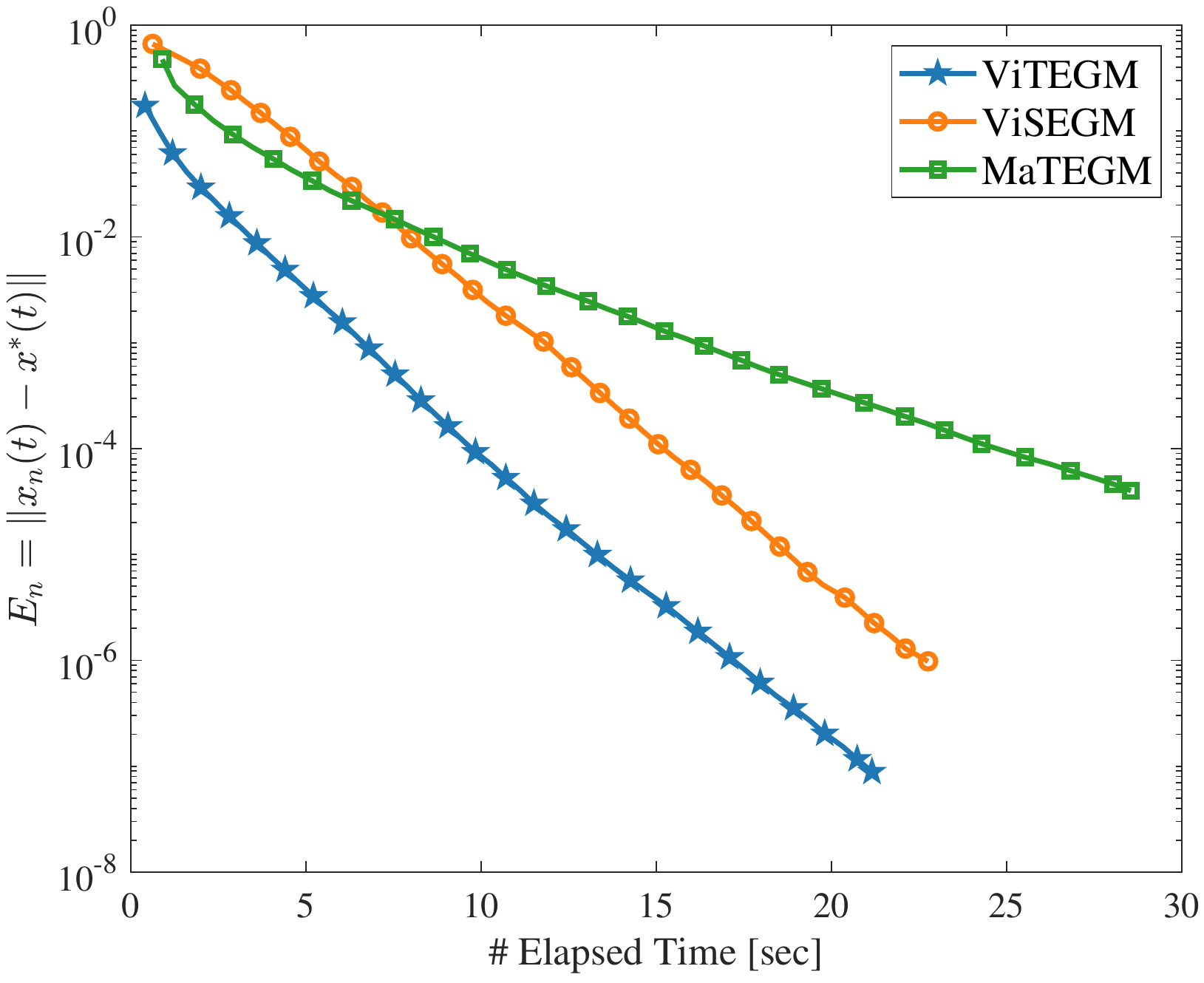}}
\caption{Numerical results for Example \ref{ex3}}
\label{ex3_res}
\end{figure}
\end{example}

\begin{remark}
\begin{enumerate}[label=(\arabic*)]
\item From Figs.\ref{ex1_data5}--\ref{ex3_res}, we can see that our proposed algorithm converges quickly and has better computational performance than the existing algorithms. In addition, these results are independent of the selection of initial values and the size of dimensions. Therefore, our algorithm is robust.
\item It should be emphasized that Algorithm~\eqref{MaTEGM} needs to spend more running time to achieve the same error accuracy because it uses Armoji-type rules to automatically update the step size, and this update criterion requires to calculate the value of operator~$ A $ many times in each iteration. However, our proposed Algorithm~\ref{alg1} uses previously known information to update the step size by a simple calculation in each iteration, which makes it converge faster.
\item It is noted that operator $ \mathcal{A} $ is pseudomonotone or strongly pseudomonotone in our numerical experiments. At this point, algorithms~\cite{THna,YLNA,FQOPT2020} for solving monotone \eqref{VIP} will not be available. Therefore, our proposed algorithm is more applicable for practical applications.
\end{enumerate}
\end{remark}

Next, we use our proposed Algorithm~\ref{alg1} to solve the \eqref{VIP} that appears in optimal control problems. Recently, many scholars have proposed different methods to solve it. We recommend readers to refer to \cite{PV,VS2019,HSM2020} for the algorithms and detailed description of the problem.
\begin{example}[Control of a harmonic oscillator, see~\cite{PSV}]\label{ex41}
\[
\begin{aligned}
\text{minimize} \;\;\; &x_{2}(3 \pi)\\
\text{subject to} \;\;\; & \dot{x}_{1}(t)=x_{2}(t)\,, \\
\;\;\;& \dot{x}_{2}(t) =-x_{1}(t)+u(t), \;\;\forall t \in[0,3 \pi]\,, \\
\;\;\;& x(0) =0\,, \\
\;\;\;& u(t) \in[-1,1]\,.
\end{aligned}
\]
The exact optimal control of Example~\ref{ex41} is known:
\[
u^{*}(t)=\left\{\begin{aligned}
1, \quad & \text { if }\; t \in[0, \pi / 2) \cup(3 \pi / 2,5 \pi / 2)\,; \\
-1, \quad & \text { if }\; t \in(\pi / 2,3 \pi / 2) \cup(5 \pi / 2,3 \pi]\,.
\end{aligned}\right.
\]
Our parameters are set as follows:
\[
N=100, \phi=0.1, \gamma_{1}=0.4,  \delta=0.3,  \epsilon_{n}=\frac{10^{-4}}{(n+1)^2}, \varphi_{n}=\frac{10^{-4}}{n+1}, f(x)=0.1x\,.
\]
The initial controls $ u_{0}(t)=u_{1}(t) $ are randomly generated in $ [-1,1] $, and the stopping criterion is $\left\|u_{n+1}-u_{n}\right\| \leq 10^{-4} $ or maximum iteration $ 1000 $ times. After $ 122 $ iterations, Algorithm~\ref{alg1} took $ 0.059839 $ seconds to reach the required error accuracy. Fig.~\ref{ex3_fig} shows the approximate optimal control and the corresponding trajectories of Algorithm~\ref{alg1}.
\begin{figure}[htbp]
\centering
\subfigure[Initial and optimal controls]{
\includegraphics[width=0.45\textwidth]{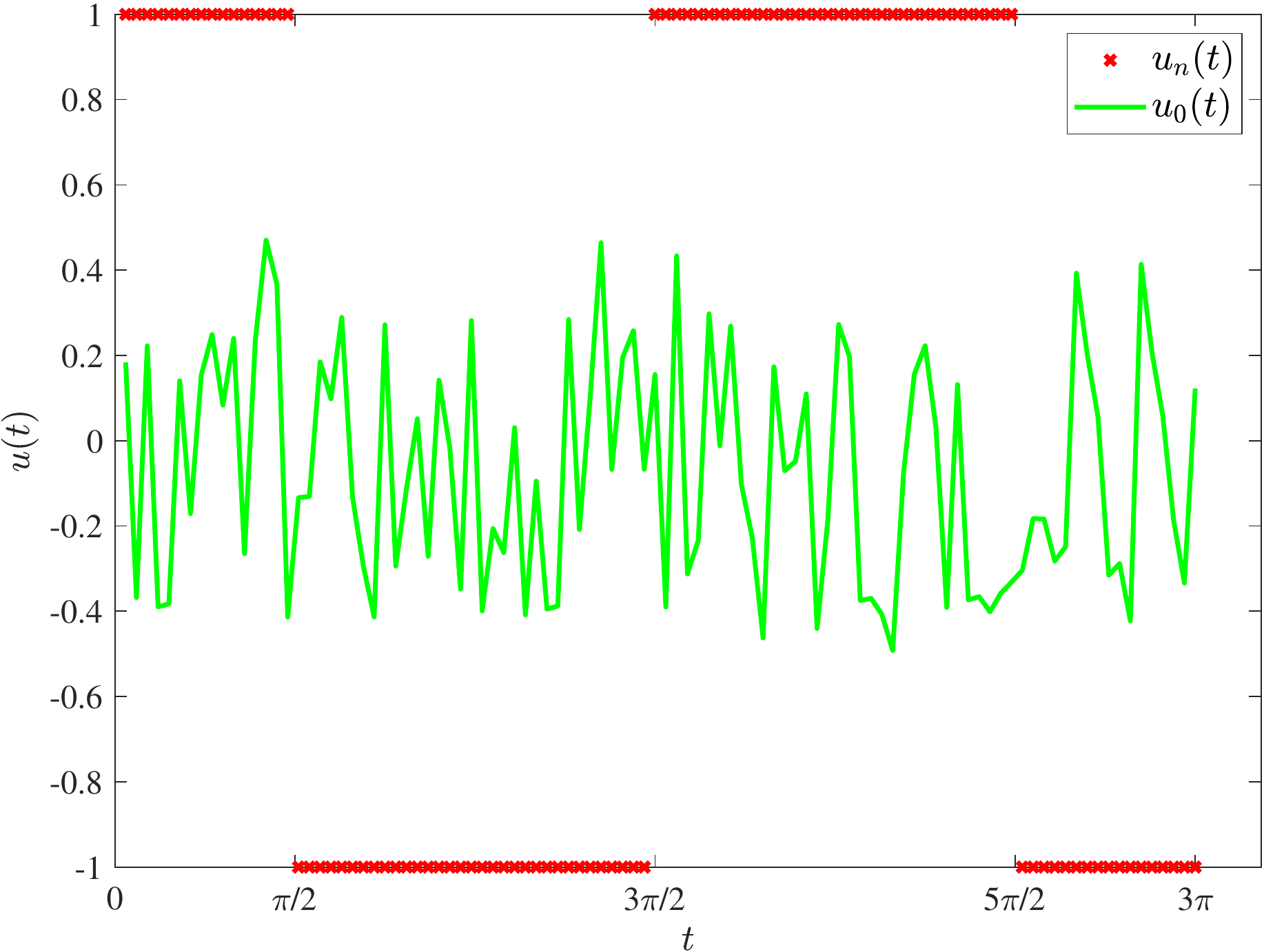}}
\subfigure[Optimal trajectories]{
\includegraphics[width=0.45\textwidth]{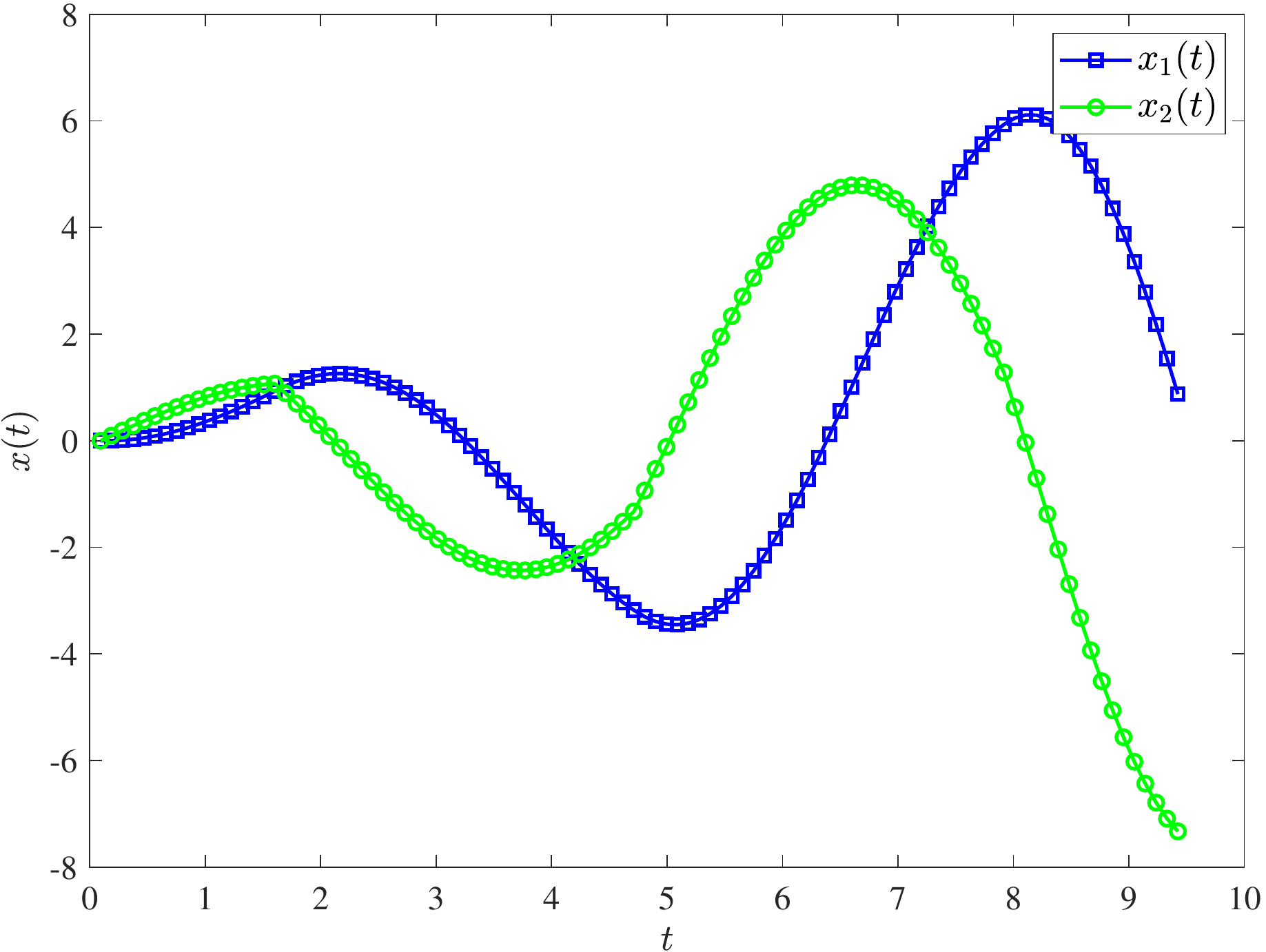}}
\caption{Numerical results for Example~\ref{ex41}}
\label{ex3_fig}
\end{figure}
\end{example}

We now consider an example in which the terminal function is not linear.
\begin{example}[See~\cite{BP2007}]\label{ex42}
\[
\begin{aligned}
\text{minimize} \;\;\; & -x_{1}(2)+\left(x_{2}(2)\right)^{2}\,, \\
\text{subject to} \;\;\; & \dot{x}_{1}(t)=x_{2}(t)\,, \\
\;\;\; & \dot{x}_{2}(t)=u(t), \;\;\forall t \in[0,2]\,, \\
\;\;\; & x_{1}(0)=0, \;\;x_{2}(0)=0\,, \\
\;\;\; & u(t) \in[-1,1]\,.
\end{aligned}
\]
The exact optimal control of Example~\ref{ex42} is
\[
u^{*}(t)=\left\{\begin{aligned}
1 \quad& \text { if }\; t \in[0,1.2)\,; \\
-1 \quad& \text { if }\; t \in(1.2,2]\,.
\end{aligned}\right.
\]
In this example, the parameters of our algorithm are set the same as in Example~\ref{ex41}. After the maximum allowable iteration of $1000$ times, Algorithm~\ref{alg1} took $0.39932$ seconds, but the required error accuracy was not achieved. Reaching the allowable error range may require more iterations. The approximate optimal control and the corresponding trajectories of Algorithm~\ref{alg1} are plotted in Fig.~\ref{ex4_fig}.
\begin{figure}[htbp]
\centering
\subfigure[Initial and optimal controls]{
\includegraphics[width=0.45\textwidth]{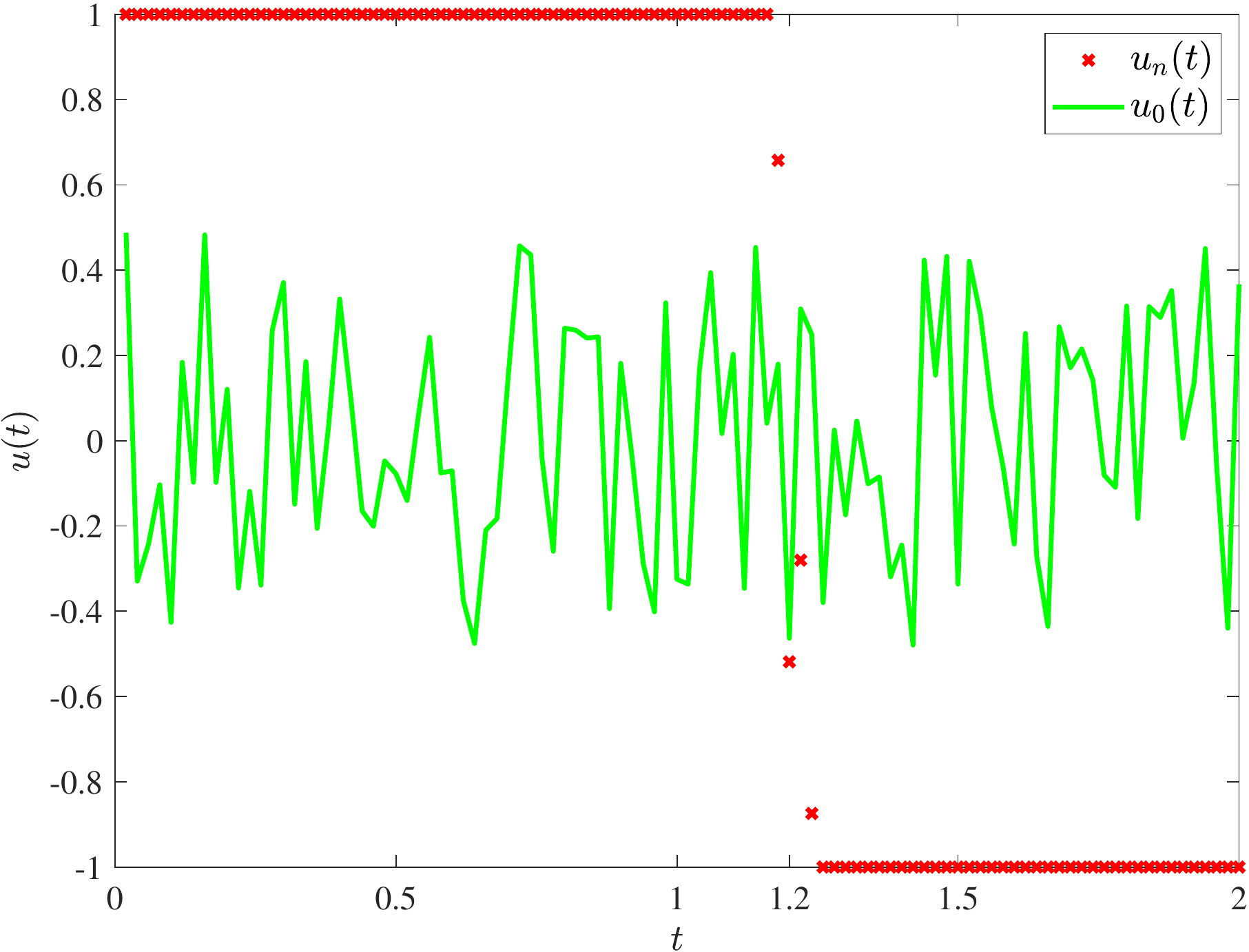}}
\subfigure[Optimal trajectories]{
\includegraphics[width=0.45\textwidth]{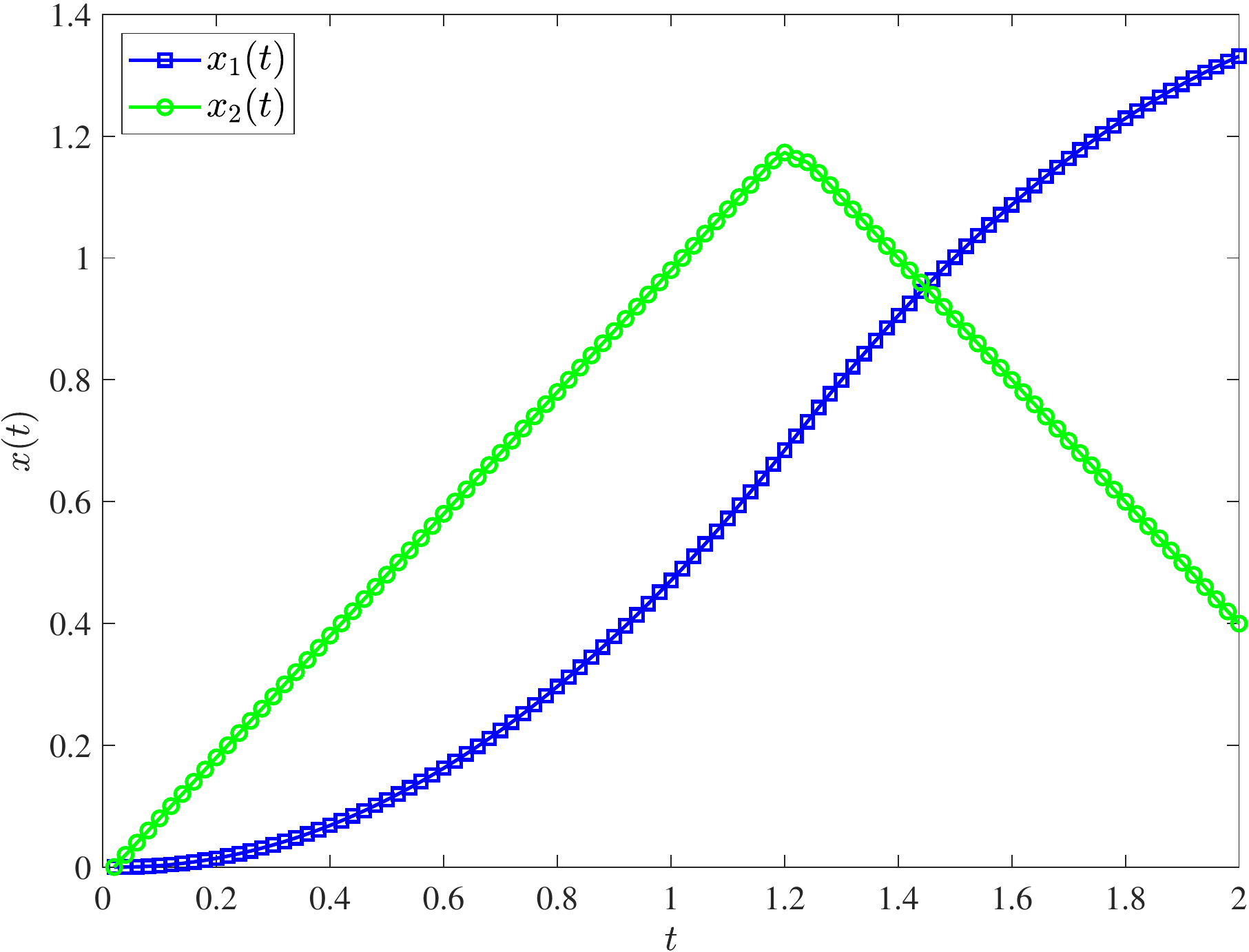}}
\caption{Numerical results for Example~\ref{ex42}}
\label{ex4_fig}
\end{figure}
\end{example}

\begin{remark}
As can be seen from Examples~\ref{ex41} and \ref{ex42}, the algorithm proposed in this paper can work well on optimal control problems.  It should be pointed out that our proposed algorithm can work better when the terminal function is linear rather than nonlinear (cf. Figs.~\ref{ex3_fig} and \ref{ex4_fig}).
\end{remark}

\section{The conclusion}\label{sec5}
In this paper, based on the inertial method, the Tseng's extragradient method and the viscosity method, we introduced a new extragradient algorithm to solve the pseudomonotone variational inequality   in a Hilbert space. The main benefit of the suggested method is that only one projection needs to be calculated in each iteration. The convergence of the algorithm was proved without the prior information of the Lipschitz constant of the mapping. Moreover, our algorithm adds an inertial term, which greatly improves the convergence speed of the algorithm. Our numerical experiments showed that the proposed algorithm improves some results of the existing algorithms in the literature. As an application, the variational inequality problem in the optimal control problem was also studied.

\end{document}